\documentclass[12pt]{article}
\usepackage{fullpage,graphicx,psfrag,amsmath,amsfonts,verbatim,url,amssymb,amsthm}
\usepackage[small,bf]{caption}
\usepackage[font=footnotesize]{subcaption}
\usepackage{algorithm}
\usepackage[noend]{algpseudocode}
\usepackage[T1]{fontenc}
\usepackage{tikz}
\usetikzlibrary{positioning,calc}
\definecolor {processblue}{cmyk}{0.96,0,0,0}
\theoremstyle{definition}

\newcommand{\ConvCol}{\mathop{\bf Col}}
\newcommand{\ConvRow}{\mathop{\bf Row}}
\newcommand{\ConvCirc}{\mathop{\bf Circ}}
\newcommand{\ConvRev}{\mathop{\bf rev}}

\newcommand{\nnz}{\mathop{\bf nnz}}

\newcommand{\mat}{\mathop{\bf mat}}
\newcommand{\FAOsum}{\mathop{\bf sum}}
\newcommand{\FAOcopy}{\mathop{\bf copy}}

\newcommand{\vstack}{\mathop{\bf vstack}}
\newcommand{\FAOsplit}{\mathop{\bf split}}

\newcommand{\kronecker}{\raisebox{1pt}{\ensuremath{\:\otimes\:}}}
\newcommand*\mydots{{\Huge$\cdots$}}

\algrenewcommand\alglinenumber[1]{
    {\sf\footnotesize\addfontfeatures{Colour=888888,Numbers=Monospaced}#1}}
\algrenewcommand\algorithmicrequire{\textbf{Precondition:}}
\algrenewcommand\algorithmicensure{\textbf{Postcondition:}}
\mathchardef\mhyphen="2D
\renewcommand{\algorithmicrequire}{\textbf{Input:}}

\newcommand{\BEAS}{\begin{eqnarray*}}
\newcommand{\EEAS}{\end{eqnarray*}}
\newcommand{\BEA}{\begin{eqnarray}}
\newcommand{\EEA}{\end{eqnarray}}
\newcommand{\BEQ}{\begin{equation}}
\newcommand{\EEQ}{\end{equation}}
\newcommand{\BIT}{\begin{itemize}}
\newcommand{\EIT}{\end{itemize}}
\newcommand{\BNUM}{\begin{enumerate}}
\newcommand{\ENUM}{\end{enumerate}}

\newcommand{\BA}{\begin{array}}
\newcommand{\EA}{\end{array}}

\newcommand{\eg}{{\it e.g.}}
\newcommand{\ie}{{\it i.e.}}
\newcommand{\etc}{{\it etc.}}

\newcommand{\reals}{{\mbox{\bf R}}}

\newcommand{\complex}{{\mbox{\bf C}}}
\newcommand{\symm}{{\mbox{\bf S}}}  

\newcommand{\Tr}{\mathop{\bf Tr}}

\newcommand{\vect}{\mathop{\bf vec}}

\newcommand{\K}{\mathcal{K}}

\newtheorem{theorem}{Theorem}[section]

{\begin{quote}}{\end{quote}}

\makeatletter
\long\def\@makecaption#1#2{
   \vskip 9pt
   \begin{small}
   \setbox\@tempboxa\hbox{{\bf #1:} #2}
   \ifdim \wd\@tempboxa > 5.5in
        \begin{center}
        \begin{minipage}[t]{5.5in}
        \addtolength{\baselineskip}{-0.95pt}
        {\bf #1:} #2 \par
        \addtolength{\baselineskip}{0.95pt}
        \end{minipage}
        \end{center}
   \else
    \hbox to\hsize{\hfil\box\@tempboxa\hfil}
   \fi
   \end{small}\par
}
\makeatother

\newcounter{oursection}

\newcounter{lecture}

\title{Matrix-Free Convex Optimization Modeling}
\author{Steven Diamond \and Stephen Boyd}

\begin{document}
\maketitle

\begin{abstract}
We introduce a convex optimization modeling framework that
transforms a convex optimization problem expressed in a form natural
and convenient for the user into an equivalent cone
program in a way that preserves fast linear transforms in
the original problem.
By representing linear functions in the
transformation process not as matrices, but
as graphs that encode composition of linear operators,
we arrive at a matrix-free cone program, \ie,
one whose data matrix is represented by a linear operator and
its adjoint.
This cone program can then be solved by a matrix-free cone solver.
By combining the matrix-free modeling framework and cone solver,
we obtain a general method for efficiently solving convex
optimization problems involving fast linear transforms.
\end{abstract}

\clearpage
\tableofcontents
\clearpage

\section{Introduction}\label{intro_sec}

Convex optimization modeling systems like
YALMIP \cite{Lofberg:04}, CVX \cite{cvx}, CVXPY \cite{cvxpy_paper},
and Convex.jl \cite{cvxjl} provide an automated
framework for converting a convex optimization problem expressed in a
natural human-readable form into the standard form required by a solver,
calling the solver, and transforming the solution back to the
human-readable form.
This allows users to form and solve convex optimization problems
quickly and efficiently.
These systems easily handle problems with a few thousand variables,
as well as much larger problems (say, with hundreds of thousands of variables)
with enough sparsity structure, which generic solvers can exploit.

The overhead of the problem transformation, and the additional
variables and constraints introduced in the transformation process,
result in longer solve times than can be obtained
with a custom algorithm tailored specifically for the particular problem.
Perhaps surprisingly, the additional solve time (compared to
a custom solver) for a modeling system coupled to
a generic solver is often not as much
as one might imagine, at least for modest sized problems.
In many cases the convenience of easily expressing the problem makes up
for the increased solve time using a convex optimization modeling system.

Many convex optimization problems in applications like
signal and image processing, or medical imaging,
involve hundreds of thousands or many millions
of variables, and so are well out of the range that current modeling
systems can handle.
There are two reasons for this. First, the
standard form problem that would be created is too large to store on a single
machine, and second, even if it could be stored, standard interior-point
solvers would be too slow to solve it.
Yet many of these problems are readily solved on a single machine
by custom solvers, which exploit fast linear transforms in the problems.
The key to these custom solvers is to directly use
the fast transforms, never forming the associated matrix.
For this reason these algorithms are
sometimes referred to as \emph{matrix-free} solvers.

The literature on matrix-free solvers in signal and image processing
is extensive; see, \eg,
\cite{TB:09,doi:10.1137/080716542,Ch:06,CP:11,doi:10.1137/080725891,ZCB:07}.
There has been particular interest in matrix-free solvers for LASSO and
basis pursuit denoising problems
\cite{doi:10.1137/080716542,doi:10.1137/S1064827596304010,FGZ:12,fig:07,4407767,doi:10.1137/080714488}.
Matrix-free solvers have also been developed for specialized
control problems \cite{vandenberghe1993polynomial,VaB:95}.
The most general matrix-free solvers target semidefinite programs
\cite{KS:09} or quadratic programs and
related problems \cite{PDCO,GondzioMatFree:12}.
The software closest to a convex optimization modeling system for
matrix-free problems is TFOCS,
which allows users to specify many types of convex problems and solve
them using a variety of matrix-free first-order methods \cite{BCG:11}.

To better understand the advantages of matrix-free solvers,
consider the nonnegative deconvolution problem
\begin{equation}\label{nonneg_deconv_prob}
\begin{array}{ll}
\mbox{minimize} & \|c*x - b\|^2 \\
\mbox{subject to} & x \geq 0,
\end{array}
\end{equation}
where $x \in \reals^n$ is the optimization variable, $c \in \reals^n$
and $b \in \reals^{2n-1}$ are problem data, and
$*$ denotes convolution.
Note that the problem data has size $O(n)$.
There are many custom matrix-free methods for efficiently solving this problem,
with $O(n)$ memory and a few hundred iterations, each of which costs $O(n \log n)$
floating point operations (flops).
It is entirely practical to solve instances of this problem
of size $n=10^7$ on a single computer
\cite{Krishnaprasad1977339,LLS:04}.

Existing convex optimization modeling systems fall far short of the
efficiency of matrix-free solvers on problem (\ref{nonneg_deconv_prob}).
These modeling systems target a standard form in which a problem's linear
structure is represented as a sparse matrix.
As a result, linear functions must be converted into explicit matrix multiplication.
In particular, the operation of convolving by $c$ will be represented as multiplication by a
$(2n-1) \times n$ Toeplitz matrix $C$.
A modeling system will thus transform problem (\ref{nonneg_deconv_prob}) into the problem
\begin{equation}\label{nonneg_deconv_matrix_prob}
\begin{array}{ll}
\mbox{minimize} & \|Cx - b\|^2 \\
\mbox{subject to} & x \geq 0,
\end{array}
\end{equation}
as part of the conversion into standard form.

Once the transformation from (\ref{nonneg_deconv_prob}) to (\ref{nonneg_deconv_matrix_prob}) has taken place,
there is no hope of solving the problem efficiently.
The explicit matrix representation of $C$ requires $O(n^2)$ memory.
A typical interior-point method for solving the transformed
problem will take a few tens of iterations, each requiring $O(n^3)$ flops.
For this reason existing convex optimization modeling systems will struggle to
solve instances of problem (\ref{nonneg_deconv_prob}) with $n=10^4$,
and when they are able to solve the problem,
they will be dramatically slower than custom matrix-free methods.

The key to matrix-free methods is to exploit fast algorithms for evaluating a linear function and its adjoint.
We call an implementation of a linear function that allows us to evaluate
the function and its adjoint a \emph{forward-adjoint oracle} (FAO).
In this paper we describe a new algorithm for converting convex optimization problems into standard form while preserving fast linear functions. (A
preliminary version of this paper appeared in \cite{DB:ICCV}.)
This yields a convex optimization modeling system that can take advantage of fast linear transforms,
and can be used to solve large problems
such as those arising in image and signal processing and other areas,
with millions of variables.
This allows users to rapidly prototype and implement new convex optimization based methods for
large-scale problems.
As with current modeling systems, the goal is not to attain (or beat)
the performance of a custom solver tuned for the specific problem;
rather it is to make the specification of the problem straightforward,
while increasing solve times only moderately.


The outline of our paper is as follows.
In \S\ref{fao_sec} we give many examples of useful FAOs.
In \S\ref{composition_sec} we explain how to compose FAOs so that
we can efficiently evaluate the composition and its adjoint.
In \S\ref{cone_prog_and_solvers_sec} we describe cone programs,
the standard intermediate-form representation of a convex problem,
and solvers for cone programs.
In \S\ref{canon_sec} we describe our algorithm for converting convex
optimization problems into equivalent cone programs
while preserving fast linear transforms.
In \S\ref{numerical_sec} we report numerical results for the
nonnegative deconvolution problem (\ref{nonneg_deconv_prob})
and a special type of linear program,
for our implementation of the abstract ideas in the paper,
using versions of the existing cone solvers
SCS \cite{SCSpaper} and
POGS \cite{fougner2015parameter} modified to be matrix-free.
(The details of these modifications will be described elsewhere.)
Even with our simple, far from optimized matrix-free cone solver,
we demonstrate scaling to problems far larger than those that
can be solved by generic methods (based on sparse matrices), with
acceptable performance loss compared to specialized custom algorithms
tuned to the problems.

We reserve certain details of our matrix-free canonicalization algorithm
for the appendix.
In \S\ref{equivalence_sec} we explain the precise sense in which
the cone program output by our algorithm is equivalent to
the original convex optimization problem.
In \S\ref{sparse_mat_sec} we describe how existing modeling systems
generate a sparse matrix representation of the cone program.
The details of this process have never been published,
and it is interesting to compare with our algorithm.

\section{Forward-adjoint oracles}\label{fao_sec}

\subsection{Definition}\label{fao_def_sec}

A general linear function $f : \reals^n \to \reals^m$ can be
represented on a computer as a dense matrix
$A \in \reals^{m \times n}$ using $O(mn)$ bytes.
We can evaluate $f(x)$ on an input $x \in \reals^n$ in $O(mn)$ flops by
computing the matrix-vector multiplication $Ax$.
We can likewise evaluate the adjoint $f^*(y) =A^Ty$ on an input
$y \in \reals^m$ in $O(mn)$ flops by computing $A^Ty$.

Many linear functions arising in applications have structure
that allows the function and its adjoint to be evaluated in fewer than $O(mn)$ flops
or using fewer than $O(mn)$ bytes of data.
The algorithms and data structures used to evaluate such a function
and its adjoint can differ wildly.
It is thus useful to abstract away the details and view linear functions
as \emph{forward-adjoint oracles} (FAOs), \ie,
a tuple $\Gamma = (f, \Phi_f, \Phi_{f^*})$ where $f$ is a linear function,
$\Phi_f$ is an algorithm for evaluating $f$,
and $\Phi_{f^*}$ is an algorithm for evaluating $f^*$.
We use $n$ to denote the size of $f$'s input and
$m$ to denote the size of $f$'s output.

While we focus on linear functions from $\reals^n$ into $\reals^m$,
the same techniques can be used to handle linear functions involving
complex arguments or values, \ie, from
$\complex^n$ into $\complex^m$, from
$\reals^n$ into $\complex^m$, or from $\complex^n$ into $\reals^m$,
using the standard embedding of complex $n$-vectors into real
$2n$-vectors.
This is useful for problems in which complex data arise naturally
(\eg, in signal processing and communications),
and also in some cases that involve only real data, where
complex intermediate results appear (typically via an FFT).

\subsection{Vector mappings}\label{slf_example_sec}

We present a variety of FAOs for functions that
take as argument, and return, vectors.

\paragraph{Scalar multiplication.}
Scalar multiplication by $\alpha \in \reals$ is represented
by the FAO $\Gamma = (f, \Phi_f, \Phi_{f^*})$,
where
$f: \reals^{n} \to \reals^{n}$ is given by $f(x) = \alpha x$.
The adjoint $f^*$ is the same as $f$.
The algorithms $\Phi_f$ and $\Phi_{f^*}$ simply scale the input,
which requires $O(m+n)$ flops and
$O(1)$ bytes of data to store $\alpha$.
Here $m=n$.

\paragraph{Multiplication by a dense matrix.}
Multiplication by a dense matrix
$A \in \reals^{m \times n}$
is represented by the FAO
$\Gamma = (f, \Phi_f, \Phi_{f^*})$,
where $f(x) = Ax$.
The adjoint $f^*(u) = A^Tu$ is also multiplication by a dense matrix.
The algorithms $\Phi_f$ and $\Phi_{f^*}$ are
the standard dense matrix multiplication algorithm.
Evaluating $\Phi_f$ and $\Phi_{f^*}$ requires $O(mn)$ flops and
$O(mn)$ bytes of data to store $A$ and $A^T$.

\paragraph{Multiplication by a sparse matrix.}
Multiplication by a sparse matrix $A \in \reals^{m \times n}$,
\ie, a matrix with many zero entries,
is represented by the FAO
$\Gamma = (f, \Phi_f, \Phi_{f^*})$, where $f(x) = Ax$.
The adjoint $f^*(u) = A^Tu$ is also multiplication by a sparse matrix.
The algorithms $\Phi_f$ and $\Phi_{f^*}$ are the standard algorithm
for multiplying by a sparse matrix in (for example)
compressed sparse row format.
Evaluating $\Phi_f$ and $\Phi_{f^*}$ requires $O(\nnz(A))$ flops and
$O(\nnz(A))$ bytes of data to store $A$ and $A^T$,
where $\nnz$ is the number of nonzero elements in a sparse matrix \cite[Chap.~2]{Davis:2006:DMS:1196434}.


\paragraph{Multiplication by a low-rank matrix.}
Multiplication by a matrix $A \in \reals^{m \times n}$
with rank $k$, where $k \ll m$ and $k \ll n$, is represented by the FAO
$\Gamma = (f, \Phi_f, \Phi_{f^*})$,
where $f(x) = Ax$.
The matrix $A$ can be factored as $A = BC$,
where $B \in \reals^{m \times k}$ and $C \in \reals^{k \times n}$.
The adjoint $f^*(u) = C^TB^Tu$ is also multiplication by a rank $k$ matrix.
The algorithm $\Phi_f$ evaluates $f(x)$ by first evaluating $z = Cx$ and then evaluating $f(x) = Bz$.
Similarly, $\Phi_{f^*}$ multiplies by $B^T$ and then $C^T$.
The algorithms $\Phi_f$ and $\Phi_{f^*}$ require $O(k(m+n))$ flops
and use $O(k(m+n))$ bytes of data to store
$B$ and $C$ and their transposes.
Multiplication by a low-rank matrix occurs in many applications,
and it is often possible to approximate multiplication by
a full rank matrix with multiplication by a low-rank one,
using the singular value decomposition or methods such as
sketching \cite{Liberty:2013:SDM:2487575.2487623}.



\paragraph{Discrete Fourier transform.}
The discrete Fourier transform (DFT) is represented by the FAO
$\Gamma = (f, \Phi_{f}, \Phi_{f^*})$,
where $f: \reals^{2p} \to \reals^{2p}$ is given by
\[
\begin{array}{ccc}
f(x)_k &= &\frac{1}{\sqrt {p}} \sum_{j=1}^p
\Re\left(\omega_p^{(j-1)(k-1)}\right)x_j - \Im\left(\omega_p^{(j-1)(k-1)}\right) x_{j+p} \\
f(x)_{k+p} &= &\frac{1}{\sqrt {p}} \sum_{j=1}^p
\Im\left(\omega_p^{(j-1)(k-1)}\right)x_j + \Re\left(\omega_p^{(j-1)(k-1)}\right) x_{j+p}
\end{array}
\]
for $k=1,\ldots, p$. Here $\omega_p = e^{-2 \pi i/p}$.
The adjoint $f^*$ is the inverse DFT.
The algorithm $\Phi_{f}$ is the fast Fourier transform (FFT),
while $\Phi_{f^*}$ is the inverse FFT.
The algorithms can be evaluated in $O((m+n)\log(m+n))$ flops,
using only $O(1)$ bytes of data to store the dimensions of $f$'s input
and output \cite{cooley1965algorithm,9780898712858}.
Here $m=n=2p$.
There are many fast transforms derived from the DFT,
such as the discrete Hartley transform \cite{1457236}
and the discrete sine and cosine transforms \cite{ANR:74,Mar:94},
with the same computational complexity as the FFT.

\paragraph{Convolution.}
Convolution with a kernel $c \in \reals^p$ is
defined as $f: \reals^n \to \reals^m$, where
\begin{equation}\label{conv_eq}
f(x)_k = \sum_{i+j=k+1}c_i x_j, \quad k=1, \ldots, m.
\end{equation}
Different variants of convolution restrict the indices $i,j$ to different ranges,
or interpret vector elements outside their natural ranges as zero or using
periodic (circular) indexing.

Standard (column) convolution takes $m=n+p-1$, and
defines $c_i$ and $x_j$ in~(\ref{conv_eq})
as zero when the index is ouside its range.
In this case the associated matrix
$\ConvCol(c) \in \reals^{n+p-1 \times n}$
is Toeplitz, with each column a shifted
version of $c$:
\[
\ConvCol(c) = \left[\begin{array}{ccc}
c_1 & &  \\
c_2 & \ddots & \\
\vdots & \ddots & c_1 \\
c_{p} &   & c_2 \\
    &  \ddots &  \vdots \\
    &         & c_{p}
\end{array}\right].
\]

Another standard form, row convolution, restricts the indices in
(\ref{conv_eq}) to the range $k=p,\ldots,n$.
For simplicity we assume that $n \geq p$.
In this case the associated matrix
$\ConvRow(c) \in \reals^{n-p+1 \times n}$
is Toeplitz, with each row a shifted version of $c$,
in reverse order:
\[
\ConvRow(c) = \left[\begin{array}{cccccc}
c_{p} & c_{p-1} & \hdots & c_1 &  & \\
        & \ddots & \ddots &     & \ddots & \\
        &        & c_{p} & c_{p-1} &\hdots & c_1
\end{array}\right].
\]
The matrices $\ConvCol(c)$ and $\ConvRow(c)$ are related by the equalities
\[
\ConvCol(c)^T = \ConvRow(\ConvRev(c)), \qquad
\ConvRow(c)^T = \ConvCol(\ConvRev(c)),
\]
where $\ConvRev(c)_k = c_{p-k+1}$ reverses the order of
the entries of $c$.

Yet another variant on convolution is circular convolution,
where we take $p=n$ and interpret the entries of vectors
outside their range modulo $n$.
In this case the associated matrix
$\ConvCirc(c) \in \reals^{n \times n}$
is Toeplitz, with each column and row a (circularly) shifted
version of $c$:
\[
\ConvCirc(c) = \left[\begin{array}{cccccc}
c_{1}  & c_{n} & c_{n-1} & \hdots & \hdots  & c_2 \\
c_{2}  & c_1 &  c_{n} & \ddots &   & \vdots \\
c_{3}  & c_2 & \ddots & \ddots & \ddots & \vdots \\
\vdots  & \ddots & \ddots & \ddots & c_{n} & c_{n-1} \\
\vdots  &        & \ddots & c_2 & c_1 & c_{n} \\
c_{n} & \hdots & \hdots & c_3 & c_2 & c_1
\end{array}\right].
\]

Column convolution with $c \in \reals^p$ is represented by the FAO
$\Gamma = (f, \Phi_{f}, \Phi_{f^*})$,
where $f: \reals^n \to \reals^{n+p-1}$ is given by
$f(x) = \ConvCol(c)x$.
The adjoint $f^*$ is row convolution with $\ConvRev(c)$,
\ie, $f^*(u) = \ConvRow(\ConvRev(c))u$.
The algorithms $\Phi_f$ and $\Phi_{f^*}$ are given in
algorithms \ref{col_conv_alg} and \ref{row_conv_alg},
and require $O((m+n+p)\log(m+n+p))$ flops.
Here $m = n + p - 1$.
If the kernel is small (\ie, $p \ll n$),
$\Phi_f$ and $\Phi_{f^*}$ instead evaluate (\ref{conv_eq})
directly in $O(np)$ flops.
In either case, the algorithms $\Phi_f$ and $\Phi_{f^*}$ use $O(p)$ bytes
of data to store $c$ and $\ConvRev(c)$ \cite{Co:69,9780898712858}.

\begin{algorithm}
\caption{Column convolution $c*x$.}\label{col_conv_alg}
\begin{algorithmic}
\Require{$c \in \reals^p$ is a length $p$ array.
$x \in \reals^n$ is a length $n$ array.
$y \in \reals^{n+p-1}$ is a length $n+p-1$ array.}
\Statex
\State Extend $c$ and $x$ into length $n+p-1$ arrays
by appending zeros.
\State $\hat{c} \gets \text{FFT of } c$.
\State $\hat{x} \gets \text{FFT of } x$.
\For {$i=1,\ldots,n+p-1$}
    \State $y_i \gets \hat{c}_i\hat{x}_i$.
\EndFor
\State $y \gets \text{inverse FFT of } y$.
\Statex
\Ensure{$y = c*x$.}
\end{algorithmic}
\end{algorithm}

\begin{algorithm}
\caption{Row convolution $c*u$.}\label{row_conv_alg}
\begin{algorithmic}
\Require{$c \in \reals^p$ is a length $p$ array.
$u \in \reals^{n+p-1}$ is a length $n+p-1$ array.
$v \in \reals^n$ is a length $n$ array.}
\Statex
\State Extend $\ConvRev(c)$ and $v$ into
length $n+p-1$ arrays by appending zeros.
\State $\hat{c} \gets \text{inverse FFT of zero-padded} \ConvRev(c)$.
\State $\hat{u} \gets \text{FFT of } u$.
\For {$i=1,\ldots,n+p-1$}
    \State $v_i \gets \hat{c}_i\hat{u}_i$.
\EndFor
\State $v \gets \text{inverse FFT of } v$.
\State Reduce $v$ to a length $n$ array by removing the
last $p-1$ entries.
\Statex
\Ensure{$v = c*u$.}
\end{algorithmic}
\end{algorithm}

\begin{algorithm}
\caption{Circular convolution $c*x$.}\label{circ_conv_alg}
\begin{algorithmic}
\Require{$c \in \reals^n$ is a length $n$ array.
$x \in \reals^n$ is a length $n$ array.
$y \in \reals^n$ is a length $n$ array.}
\Statex
\State $\hat{c} \gets \text{FFT of } c$.
\State $\hat{x} \gets \text{FFT of } x$.
\For {$i=1,\ldots,n$}
    \State $y_i \gets \hat{c}_i\hat{x}_i$.
\EndFor
\State $y \gets \text{inverse FFT of } y$.
\Statex
\Ensure{$y = c*x$.}
\end{algorithmic}
\end{algorithm}

Circular convolution with $c \in \reals^n$ is represented by the FAO
$\Gamma = (f, \Phi_{f}, \Phi_{f^*})$,
where $f: \reals^n \to \reals^n$ is given by $f(x) = \ConvCirc(c)x$.
The adjoint $f^*$ is circular convolution with
\[
\tilde c = \left[\begin{array}{c}
c_1 \\
c_{n} \\
c_{n-1} \\
\vdots \\
c_{2}
\end{array}\right].
\]
The algorithms $\Phi_f$ and $\Phi_{f^*}$ are given in
algorithm \ref{circ_conv_alg},
and require $O((m+n)\log(m+n))$ flops.
The algorithms $\Phi_f$ and $\Phi_{f^*}$ use $O(m+n)$
bytes of data to store $c$ and $\tilde c$ \cite{Co:69,9780898712858}.
Here $m=n$.



\paragraph{Discrete wavelet transform.}
The discrete wavelet transform (DWT) for orthogonal wavelets
is represented by the FAO $\Gamma = (f, \Phi_f, \Phi_{f^*})$,
where the function $f:\reals^{2^p} \to \reals^{2^p}$ is given by
\begin{equation}\label{DWT}
f(x) =
\left[\begin{array}{cc}
D_{1}G_{1} & \\
D_{1}H_{1} & \\
 & I_{2^p-2}
\end{array}\right]
\cdots
\left[\begin{array}{cc}
D_{p-1}G_{p-1} & \\
D_{p-1}H_{p-1} & \\
 & I_{2^{p-1}}
\end{array}\right]
\left[\begin{array}{c}
D_pG_p \\
D_pH_p
\end{array}\right]x,
\end{equation}
where $D_k \in \reals^{2^{k-1} \times 2^k}$ is defined such that
$(D_kx)_i = x_{2i}$
and the matrices $G_k \in \reals^{2^k \times 2^k}$ and
$H_k \in \reals^{2^k \times 2^k}$
are given by
\[
G_k = \ConvCirc\left(
\left[\begin{array}{c}
g \\
0 \\
\end{array}\right]\right), \qquad
H_k = \ConvCirc\left(
\left[\begin{array}{c}
h \\
0 \\
\end{array}\right]\right).
\]
Here $g \in \reals^q$ and $h \in \reals^q$ are low and high pass
filters, respectively, that parameterize the DWT.
The adjoint $f^*$ is the inverse DWT.
The algorithms $\Phi_f$ and $\Phi_f^*$ repeatedly convolve
by $g$ and $h$, which requires $O(q(m+n))$ flops and uses
$O(q)$ bytes to store $h$ and $g$ \cite{Mal:89}.
Here $m=n=2^p$.
Common orthogonal wavelets include the Haar wavelet and the
Daubechies wavelets \cite{Dau:88,daubechies1992ten}.
There are many variants on the particular DWT described here.
For instance, the product in (\ref{DWT}) can be terminated after
fewer than $p-1$ multiplications by $G_k$ and $H_k$ \cite{ripples},
$G_k$ and $H_k$ can be defined as a different type of convolution matrix,
or the filters $g$ and $h$ can be different lengths,
as in biorthogonal wavelets \cite{CDF:92}.





\paragraph{Discrete Gauss transform.}
The discrete Gauss transform (DGT) is represented by the
FAO $\Gamma = (f_{Y,Z,h}, \Phi_{f}, \Phi_{f^*})$,
where the function $f_{Y,Z,h}: \reals^n \to \reals^m$
is parameterized by $Y \in \reals^{m \times d}$,
$Z \in \reals^{n \times d}$, and $h > 0$.
The function $f_{Y,Z,h}$ is given by
\[
f_{Y,Z,h}(x)_i = \sum_{j=1}^n\exp(-\|y_{i}-z_{j}\|^2/h^2)x_j,
\quad i=1,\ldots,m,
\]
where $y_i \in \reals^d$ is the $i$th column of $Y$ and
$z_j \in \reals^d$ is the $j$th column of $Z$.
The adjoint of $f_{Y,Z,h}$ is the DGT $f_{Z,Y,h}$.
The algorithms $\Phi_{f}$ and $\Phi_{f^*}$ are
the improved fast Gauss transform,
which evaluates $f(x)$ and $f^*(u)$
to a given accuracy in $O(d^p(m+n))$ flops.
Here $p$ is a parameter that depends on the accuracy desired.
The algorithms $\Phi_{f}$ and $\Phi_{f^*}$ use $O(d(m + n))$
bytes of data to store $Y$, $Z$, and $h$ \cite{1238383}.
An interesting application of the DGT is efficient
multiplication by a Gaussian kernel \cite{NIPS2004_2550}.

\paragraph{Multiplication by the inverse of a sparse triangular matrix.}
Multiplication by the inverse of a sparse lower triangular matrix
$L \in \reals^{n \times n}$
with
nonzero elements on its diagonal
is represented by the FAO
$\Gamma = (f, \Phi_{f}, \Phi_{f^*})$,
where $f(x) = L^{-1}x$.
The adjoint $f^*(u) = (L^{T})^{-1}u$ is multiplication by the inverse of
a sparse upper triangular matrix.
The algorithms $\Phi_f$ and $\Phi_{f^*}$ are forward and backward
substitution, respectively,
which require $O(\nnz(L))$ flops and use $O(\nnz(L))$ bytes of data
to store $L$ and $L^T$
\cite[Chap.~3]{Davis:2006:DMS:1196434}.


\paragraph{Multiplication by a pseudo-random matrix.}
Multiplication by a matrix $A \in \reals^{m \times n}$ whose
columns are given by a pseudo-random sequence
(\ie, the first $m$ values of the sequence are the first column of $A$,
the next $m$ values are the second column of $A$, \etc) is represented
by the FAO
$\Gamma = (f, \Phi_{f}, \Phi_{f^*})$,
where $f(x) = Ax$.
The adjoint $f^*(u) = A^Tu$ is multiplication by a matrix
whose rows are given by a pseudo-random sequence
(\ie, the first $m$ values of the sequence are the first row of $A^T$,
the next $m$ values are the second row of $A^T$, \etc).
The algorithms $\Phi_f$ and $\Phi_{f^*}$ are the standard dense matrix
multiplication algorithm,
iterating once over the pseudo-random sequence without storing any of
its values.
The algorithms require $O(mn)$ flops and use $O(1)$ bytes of data
to store the the seed for the pseudo-random sequence.
Multiplication by a pseudo-random matrix might appear, for example, as a
measurement ensemble in compressed sensing
\cite{Gilbert:2007:OSF:1250790.1250824}.

\paragraph{Multiplication by the pseudo-inverse of a graph Laplacian.}
Multiplication by the pseudo-inverse of a graph Laplacian matrix
$L \in \reals^{n \times n}$ is represented by the FAO
$\Gamma = (f, \Phi_{f}, \Phi_{f^*})$,
where $f(x) = L^\dagger x$.
A graph Laplacian is a symmetric matrix with nonpositive off
diagonal entries and the property $L1 = 0$,
\ie, the diagonal entry in a row is the negative sum of the off-diagonal
entries in that row.
(This implies that it is positive semidefinite.)
The adjoint $f^*$ is the same as $f$, since $L=L^T$.
The algorithms $\Phi_f$ and $\Phi_{f^*}$ are one of the fast
solvers for graph Laplacian systems that evaluate
$f(x)=f^*(x)$ to a given accuracy in around $O(\nnz(L))$ flops
\cite{Spielman:2004:NTA:1007352.1007372,
Kelner:2013:SCA:2488608.2488724,vishnoi2012laplacian}.
(The details of the computational complexity are much more involved.)
The algorithms use $O(\nnz(L))$ bytes of data to store $L$.






\subsection{Matrix mappings}\label{matrix_arg_sec}

We now consider linear functions that take as argument, or return,
matrices.
We take the standard inner product on matrices
$X,Y\in \reals^{p \times q}$,
\[
\langle X, Y \rangle =
\sum_{i=1,\ldots, p, ~j=1,\ldots, q}  X_{ij} Y_{ij} =
\Tr (X^TY).
\]
The adjoint of a linear function
$f: \reals^{p \times q} \to \reals^{s \times t}$
is then the function
$f^*:\reals^{s \times t} \to \reals^{p \times q}$
for which
\[
\Tr (f(X)^TY)  = \Tr(X^Tf^*(Y)),
\]
holds for all $X \in \reals^{p \times q}$ and
$Y \in \reals^{s \times t}$.

\paragraph{Vec and mat.}
The function $\vect: \reals^{p\times q} \to \reals^{p q}$ is represented
by the FAO $\Gamma = (f, \Phi_{f}, \Phi_{f^*})$,
where $f(X)$ converts the matrix $X \in \reals^{p\times q}$ into a
vector $y \in \reals^{p q}$ by stacking the columns.
The adjoint $f^*$ is the function
$\mat: \reals^{p q} \to \reals^{p\times q}$,
which outputs a matrix whose columns are successive slices of its vector argument.
The algorithms $\Phi_f$ and $\Phi_{f^*}$ simply reinterpret
their input as a differently shaped output in $O(1)$ flops,
using only $O(1)$ bytes of data to store the dimensions of $f$'s input
and output.

\paragraph{Sparse matrix mappings.}
Many common linear functions on and to matrices are given by a sparse
matrix multiplication of the vectorized argument,
reshaped as the output matrix.
For $X\in \reals^{p \times q}$ and $f(X)=Y\in \reals^{s \times t}$,
\[
Y = \mat (A \vect(X)).
\]
The form above describes the general linear mapping from
$\reals^{p \times q}$ to $\reals^{s \times t}$; we are interested in
cases when $A$ is sparse, \ie, has far fewer than $pqst$ nonzero entries.
Examples include extracting a submatrix, extracting the diagonal, forming
a diagonal matrix, summing the rows or columns of a matrix,
transposing a matrix, scaling its rows or columns, and so on.
The FAO representation of each such function is
$\Gamma = (f, \Phi_f, \Phi_{f^*})$,
where $f$ is given above and the adjoint is given by
\[
f^*(U) = \mat (A^T \vect(U)).
\]
The algorithms $\Phi_f$ and $\Phi_{f^*}$ are the standard algorithms for
multiplying a vector by a sparse matrix in (for example) compressed
sparse row format.
The algorithms require $O(\nnz(A))$ flops and use $O(\nnz(A))$ bytes of
data to store $A$ and $A^T$
\cite[Chap.~2]{Davis:2006:DMS:1196434}.

\paragraph{Matrix product.}
Multiplication on the left by a matrix $A \in \reals^{s \times p}$
and on the right by a matrix $B \in \reals^{q \times t}$
is represented by the FAO
$\Gamma = (f, \Phi_f, \Phi_{f^*})$,
where $f: \reals^{p \times q} \to \reals^{s \times t}$
is given by $f(X) = AXB$.
The adjoint $f^*(U) = A^T U B^T$ is also a matrix product.
There are two ways to implement $\Phi_f$ efficiently,
corresponding to different orders of operations in multiplying out $AXB$.
In one method we multiply by $A$ first and $B$ second, for a total of
$O(s(pq + qt))$ flops (assuming that $A$ and $B$ are dense).
In the other method we multiply by $B$ first and $A$ second,
for a total of $O(p(qt + st))$ flops.
The former method is more efficient if
\[
\frac{1}{t} + \frac{1}{p} < \frac{1}{s} + \frac{1}{q}.
\]
Similarly, there are two ways to implement $\Phi_{f^*}$,
one requiring $O(s(pq + qt))$ flops and the other requiring
$O(p(qt + st))$ flops.
The algorithms $\Phi_f$ and $\Phi_{f^*}$ use $O(sp + qt)$
bytes of data to store $A$ and $B$ and their transposes.
When $p=q=s=t$, the flop count for $\Phi_f$ and $\Phi_{f^*}$
simplifies to $O\left((m+n)^{1.5}\right)$ flops.
Here $m=n=pq$.
(When the matrices $A$ or $B$ are sparse, evaluating
$f(X)$ and $f^*(U)$ can be done even more efficiently.)
The matrix product function is used in Lyapunov and
algebraic Riccati inequalities and Sylvester equations,
which appear in many problems from control theory
\cite{Gardiner:1992:SSM:146847.146929,VaB:95}.

\paragraph{2-D discrete Fourier transform.}
The 2-D DFT is represented by the FAO
$\Gamma = (f, \Phi_{f}, \Phi_{f^*})$,
where $f:  \reals^{2p\times q} \to \reals^{2p\times q}$ is given by
\[
\begin{array}{ccc}
f(X)_{k\ell} &= &\frac{1}{\sqrt{pq}}\sum_{s=1}^p\sum_{t=1}^q
\Re\left(\omega_p^{(s-1)(k-1)}\omega_q^{(t-1)(\ell-1)}\right)X_{st}
- \Im\left(\omega_p^{(s-1)(k-1)}\omega_q^{(t-1)(\ell-1)}\right)X_{s+p,t} \\
f(X)_{k+p,\ell} &= &\frac{1}{\sqrt{pq}}\sum_{s=1}^p\sum_{t=1}^q
\Im\left(\omega_p^{(s-1)(k-1)}\omega_q^{(t-1)(\ell-1)}\right)X_{st}
+ \Re\left(\omega_p^{(s-1)(k-1)}\omega_q^{(t-1)(\ell-1)}\right)X_{s+p,t},
\end{array}
\]
for $k=1,\ldots,p$ and $\ell=1,\ldots,q$.
Here $\omega_{p} = e^{-2 \pi i/p}$ and $\omega_{q}=e^{-2 \pi i/q}$.
The adjoint $f^*$ is the inverse 2-D DFT.
The algorithm $\Phi_f$ evaluates $f(X)$ by first applying the FFT to each row of $X$,
replacing the row with its DFT,
and then applying the FFT to each column,
replacing the column with its DFT.
The algorithm $\Phi_{f^*}$ is analogous, but with the inverse FFT and
inverse DFT taking the role of the FFT and DFT.
The algorithms $\Phi_f$ and $\Phi_{f^*}$ require $O((m+n)\log(m+n))$ flops,
using only $O(1)$ bytes of data to store the dimensions of $f$'s input
and output \cite{Lim:1990:TSI:130247,9780898712858}.
Here $m=n=2pq$.


\paragraph{2-D convolution.}
2-D convolution with a kernel $C \in \reals^{p\times q}$ is defined as
$f: \reals^{s \times t} \to \reals^{m_1 \times m_2}$, where
\begin{equation}\label{2D_conv_eq}
f(X)_{k\ell}= \sum_{i_1+i_2=k+1, j_1+j_2=\ell+1}C_{i_1j_1}X_{i_2j_2},
\quad k=1,\ldots,m_1, \quad \ell=1,\ldots,m_2.
\end{equation}
Different variants of 2-D convolution restrict the indices $i_1,j_1$ and
$i_2,j_2$ to different ranges,
or interpret matrix elements outside their natural ranges as zero
or using periodic (circular) indexing.
There are 2-D analogues of 1-D column, row, and circular convolution.

Standard 2-D (column) convolution,
the analogue of 1-D column convolution,
takes $m_1 = s+p-1$ and $m_2 = t+q-1$,
and defines $C_{i_1j_1}$ and $X_{i_2j_2}$ in (\ref{2D_conv_eq}) as zero
when the indices are outside their range.
We can represent the 2-D column convolution $Y = C*X$ as the matrix
multiplication
\[
Y = \mat(\ConvCol(C)\vect(X)),
\]
where $\ConvCol(C) \in \reals^{(s+p-1)(t+q-1) \times st}$
is given by:
\[
\ConvCol(C) = \left[\begin{array}{ccc}
\ConvCol(c_1) & &  \\
\ConvCol(c_2) & \ddots & \\
\vdots & \ddots & \ConvCol(c_1) \\
\ConvCol(c_q) &   & \ConvCol(c_2) \\
    &  \ddots &  \vdots \\
    &         & \ConvCol(c_q)
\end{array}\right].
\]
Here $c_1,\ldots,c_q \in \reals^p$ are the columns of $C$
and $\ConvCol(c_1),\ldots,\ConvCol(c_q) \in \reals^{s+p-1 \times s}$
are 1-D column convolution matrices.


The 2-D analogue of 1-D row convolution restricts the indices in
(\ref{2D_conv_eq})
to the range $k=p,\ldots,s$ and $\ell=q,\ldots,t$.
For simplicity we assume $s\geq p$ and $t \geq q$.
The output dimensions are $m_1 = s-p+1$ and $m_2 = t-q+1$.
We can represent the 2-D row convolution $Y = C*X$ as the matrix
multiplication
\[
Y = \mat(\ConvRow(C)\vect(X)),
\]
where $\ConvRow(C) \in \reals^{(s-p+1)(t-q+1) \times st}$
is given by:
\[
\ConvRow(C) = \left[\begin{array}{cccccc}
\ConvRow(c_{q}) & \ConvRow(c_{q-1}) & \hdots & \ConvRow(c_1) &  & \\
        & \ddots & \ddots &     & \ddots & \\
        &        & \ConvRow(c_{q}) & \ConvRow(c_{q-1}) &
        \hdots & \ConvRow(c_1)
\end{array}\right].
\]
Here $\ConvRow(c_1),\ldots,\ConvRow(c_q) \in \reals^{s-p+1 \times s}$
are 1-D row convolution matrices.
The matrices $\ConvCol(C)$ and $\ConvRow(C)$ are related by the
equalities
\[
\ConvCol(C)^T = \ConvRow(\ConvRev(C)), \qquad
\ConvRow(C)^T = \ConvCol(\ConvRev(C)),
\]
where $\ConvRev(C)_{k\ell} = C_{p-k+1,q-\ell+1}$
reverses the order of of the columns of $C$ and of the entries in each
row.

In the 2-D analogue of 1-D circular convolution,
we take $p=s$ and $q=t$ and interpret the entries of matrices
outside their range modulo $s$ for the row index and modulo $t$ for the
column index.
We can represent the 2-D circular convolution $Y = C*X$ as the matrix
multiplication
\[
Y = \mat(\ConvCirc(C)\vect(X)),
\]
where $\ConvCirc(C) \in \reals^{st \times st}$
is given by:
\[
\ConvCirc(C) = \left[\begin{array}{cccccc}
\ConvCirc(c_{1})  & \ConvCirc(c_{t}) & \ConvCirc(c_{t-1}) & \hdots &
\hdots  & \ConvCirc(c_2) \\
\ConvCirc(c_{2})  & \ConvCirc(c_1) &  \ConvCirc(c_{t}) & \ddots &   & \vdots \\
\ConvCirc(c_{3})  & \ConvCirc(c_2) & \ddots & \ddots & \ddots & \vdots \\
\vdots  & \ddots & \ddots & \ddots & \ConvCirc(c_{t}) & \ConvCirc(c_{t-1}) \\
\vdots  &        & \ddots & \ConvCirc(c_2) & \ConvCirc(c_1) & \ConvCirc(c_{t}) \\
\ConvCirc(c_{t}) & \hdots & \hdots & \ConvCirc(c_3) & \ConvCirc(c_2) & \ConvCirc(c_1)
\end{array}\right].
\]
Here $\ConvCirc(c_1),\ldots,\ConvCirc(c_t) \in \reals^{s \times s}$
are 1-D circular convolution matrices.

2-D column convolution with $C \in \reals^{p\times q}$ is represented by
the FAO $\Gamma = (f, \Phi_{f}, \Phi_{f^*})$,
where $f: \reals^{s \times t} \to \reals^{s+p-1 \times t+q-1}$
is given by
\[
f(X) = \mat(\ConvCol(C)\vect(X)).
\]
The adjoint $f^*$ is 2-D row convolution with $\ConvRev(C)$, \ie,
\[
f^*(U) = \mat(\ConvRow(\ConvRev(C))\vect(U)).
\]
The algorithms $\Phi_{f}$ and $\Phi_{f^*}$ are given in
algorithms \ref{2D_col_conv_alg} and \ref{2D_row_conv_alg},
and require $O((m+n)\log(m+n))$ flops.
Here $m=(s+p-1)(t+q-1)$ and $n=st$.
If the kernel is small (\ie, $p \ll s$ and $q \ll t$),
$\Phi_{f}$ and $\Phi_{f^*}$ instead evaluate (\ref{2D_conv_eq}) directly
in $O(pqst)$ flops.
In either case, the algorithms $\Phi_{f}$ and $\Phi_{f^*}$
use $O(pq)$ bytes of data to store $C$ and $\ConvRev(C)$
\cite[Chap.~4]{9780898712858}.
Often the kernel is parameterized (\eg, a Gaussian kernel),
in which case more compact representations of $C$ and $\ConvRev(C)$ are
possible \cite[Chap.~7]{Forsyth:2002:CVM:580035}.



\begin{algorithm}
\caption{2-D column convolution $C*X$.}\label{2D_col_conv_alg}
\begin{algorithmic}
\Require{$C \in \reals^{p \times q}$ is a length $pq$ array.
$X \in \reals^{s \times t}$ is a length $st$ array.
$Y \in \reals^{s+p-1 \times t+q-1}$ is a length $(s+p-1)(t+q-1)$ array.}
\Statex
\State Extend the columns and rows of $C$ and $X$ with zeros so
$C,X \in \reals^{s+p-1 \times t+q-1}$.
\State $\hat{C} \gets \text{2-D DFT of } C$.
\State $\hat{X} \gets \text{2-D DFT of } X$.
\For {$i=1,\ldots,s+p-1$}
    \For {$j=1,\ldots,t+q-1$}
        \State $Y_{ij} \gets \hat{C}_{ij}\hat{X}_{ij}$.
    \EndFor
\EndFor
\State $Y \gets \text{inverse 2-D DFT of } Y$.
\Statex
\Ensure{$Y = C*X$.}
\end{algorithmic}
\end{algorithm}

\begin{algorithm}
\caption{2-D row convolution $C*U$.}\label{2D_row_conv_alg}
\begin{algorithmic}
\Require{$C \in \reals^{p \times q}$ is a length $pq$ array.
$U \in \reals^{s+p-1 \times t+q-1}$ is a length $(s+p-1)(t+q-1)$ array.
$V\in \reals^{s \times t}$ is a length $st$ array.}
\Statex
\State Extend the columns and rows of $\ConvRev(C)$
and $V$ with zeros so $\ConvRev(C),V \in \reals^{s+p-1 \times t+q-1}$.
\State $\hat{C} \gets \text{inverse 2-D DFT of zero-padded} \ConvRev(C)$.
\State $\hat{U} \gets \text{2-D DFT of } U$.
\For {$i=1,\ldots,s+p-1$}
    \For {$j=1,\ldots,t+q-1$}
        \State $V_{ij} \gets \hat{C}_{ij}\hat{U}_{ij}$.
    \EndFor
\EndFor
\State $V \gets \text{inverse 2-D DFT of } V$.
\State Truncate the rows and columns of $V$ so that $V \in \reals^{s \times t}$.
\Statex
\Ensure{$V = C*U$.}
\end{algorithmic}
\end{algorithm}

\begin{algorithm}
\caption{2-D circular convolution $C*X$.}\label{2D_circ_conv_alg}
\begin{algorithmic}
\Require{$C \in \reals^{s \times t}$ is a length $st$ array.
$X \in \reals^{s \times t}$ is a length $st$ array.
$Y \in \reals^{s \times t}$ is a length $st$ array.}
\Statex
\State $\hat{C} \gets \text{2-D DFT of } C$.
\State $\hat{X} \gets \text{2-D DFT of } X$.
\For {$i=1,\ldots,s$}
    \For {$j=1,\ldots,t$}
        \State $Y_{ij} \gets \hat{C}_{ij}\hat{X}_{ij}$.
    \EndFor
\EndFor
\State $Y \gets \text{inverse 2-D DFT of } Y$.
\Statex
\Ensure{$Y = C*X$.}
\end{algorithmic}
\end{algorithm}

2-D circular convolution with $C \in \reals^{s \times t}$ is represented
by the FAO $\Gamma = (f, \Phi_{f}, \Phi_{f^*})$,
where $f: \reals^{s \times t} \to \reals^{s \times t}$ is given by
\[
f(X) = \mat(\ConvCirc(C)\vect(X)).
\]
The adjoint $f^*$ is 2-D circular convolution with
\[
\tilde C = \left[\begin{array}{ccccc}
C_{1,1} & C_{1,t} & C_{1,t-1} & \hdots & C_{1,2} \\
C_{s,1} & C_{s,t} & C_{s,t-1} & \hdots & C_{s,2} \\
C_{s-1,1} & C_{s-1,t} & C_{s-1,t-1} & \hdots & C_{s-1,2} \\
\vdots & \vdots & \vdots & \ddots & \vdots \\
C_{2,1} & C_{2,t} & C_{2,t-1} & \hdots & C_{2,2} \\
\end{array}\right].
\]
The algorithms $\Phi_f$ and $\Phi_{f^*}$ are given in
algorithm \ref{2D_circ_conv_alg},
and require $O((m+n)\log(m+n))$ flops.
The algorithms $\Phi_f$ and $\Phi_{f^*}$ use $O(m+n)$
bytes of data to store $C$ and $\tilde C$ \cite[Chap.~4]{9780898712858}.
Here $m=n=st$.

\paragraph{2-D discrete wavelet transform.}
The 2-D DWT for separable, orthogonal wavelets is represented by the FAO
$\Gamma = (f, \Phi_{f}, \Phi_{f^*})$,
where $f: \reals^{2^p \times 2^p} \to \reals^{2^p \times 2^p}$
is given by
\[
f(X)_{ij} = W_k\cdots W_{p-1}W_pXW_p^TW_{p-1}^T\cdots W_k^T,
\]
where $k = \max\{\lceil\log_2(i)\rceil, \lceil\log_2(j)\rceil,1\}$
and $W_k \in \reals^{2^p \times 2^p}$ is given by
\[
W_k = \left[\begin{array}{cc}
D_kG_k & \\
D_kH_k & \\
 & I
\end{array}\right].
\]
Here $D_k$, $G_k$, and $H_k$ are defined as for the 1-D DWT.
The adjoint $f^*$ is the inverse 2-D DWT.
As in the 1-D DWT, the algorithms $\Phi_f$ and $\Phi_{f^*}$
repeatedly convolve by the filters $g \in \reals^q$ and $h \in \reals^q$,
which requires $O(q(m+n))$ flops and uses $O(q)$ bytes of data to store
$g$ and $h$ \cite{ripples}.
Here $m = n = 2^p$.
There are many alternative wavelet transforms for 2-D data; see, \eg,
\cite{doi:10.1137/05064182X,SCD:02,DV:03,Jacques20112699}.

\subsection{Multiple vector mappings}\label{multi_arg_sec}

In this section we consider linear functions that take as argument,
or return, multiple vectors.
(The idea is readily extended to the case when the
arguments or return values are matrices.)
The adjoint is defined by the inner product
\[
\langle (x_1,\ldots,x_k), (y_1,\ldots,y_k) \rangle =
\sum_{i=1}^k \langle x_i,y_i\rangle =
\sum_{i=1}^k x_i^Ty_i.
\]
The adjoint of a linear function
$f: \reals^{n_1} \times \cdots \times \reals^{n_k}
\to \reals^{m_1} \times \cdots \times \reals^{m_{\ell}}$
is then the function
$f^*: \reals^{m_1} \times \cdots \times \reals^{m_{\ell}}
\to \reals^{n_1} \times \cdots \times \reals^{n_k}$
for which
\[
\sum_{i=1}^{\ell} f(x_1,\ldots,x_k)_i^T y_i =
\sum_{i=1}^{k} x_i^T f^*(y_1,\ldots,y_{\ell})_i,
\]
holds for all $(x_1,\ldots,x_k) \in \reals^{n_1} \times \cdots \times \reals^{n_k}$
and
$(y_1,\ldots,y_{\ell}) \in \reals^{m_1} \times \cdots \times \reals^{m_{\ell}}$.
Here $f(x_1,\ldots,x_k)_i$ and $f^*(y_1,\ldots,y_{\ell})_i$
refer to the $i$th output of $f$ and $f^*$, respectively.

\paragraph{Sum and copy.}
The function
$\FAOsum : \reals^{m}\times \cdots \times \reals^{m} \to \reals^m$
with $k$ inputs is represented by the FAO
$\Gamma = (f, \Phi_{f}, \Phi_{f^*})$,
where $f(x_1,\ldots,x_k) = x_1 + \cdots + x_k$.
The adjoint $f^*$ is the function
$\FAOcopy: \reals^{m} \to \reals^{m}\times \cdots \times \reals^{m}$,
which outputs $k$ copies of its input.
The algorithms $\Phi_f$ and $\Phi_{f^*}$ require $O(m+n)$ flops to
sum and copy their input, respectively,
using only $O(1)$ bytes of data to store the dimensions of $f$'s
input and output.
Here $n=km$.

\paragraph{Vstack and split.}
The function
$\vstack: \reals^{m_1} \times \cdots \times \reals^{m_k} \to \reals^{n}$
is represented by the FAO
$\Gamma = (f, \Phi_{f}, \Phi_{f^*})$,
where $f(x_1,\ldots,x_k)$
concatenates its $k$ inputs into a single vector output.
The adjoint $f^*$ is the function
$\FAOsplit: \reals^{n} \to \reals^{m_1} \times \cdots \times \reals^{m_k}$,
which divides a single vector into $k$ separate components.
The algorithms $\Phi_f$ and $\Phi_{f^*}$ simply reinterpret
their input as a differently sized output in $O(1)$ flops,
using only $O(1)$ bytes of data to store the dimensions of $f$'s input
and output. Here $n = m = m_1 + \cdots + m_k$.

\subsection{Additional examples}
The literature on fast linear transforms
goes far beyond the preceding examples.
In this section we highlight a few notable omissions.
Many methods have been developed for matrices derived from physical
systems.
The multigrid \cite{Hackbusch1985} and algebraic multigrid
\cite{brandt1985algebraic} methods efficiently apply
the inverse of a matrix representing discretized partial differential
equations (PDEs).
The fast multipole method accelerates multiplication by matrices
representing pairwise interactions \cite{GR:73,CGR:88},
much like the fast Gauss transform \cite{greengard1991fast}.
Heirarchical matrices are a matrix format that allows fast
multiplication by the matrix and its inverse,
with applications to discretized integral operators and PDEs
\cite{Hac:99,HKS:00,Borm2003}.

Many approaches exist for factoring an invertible sparse matrix into a
product of components whose inverses can be applied efficiently,
yielding a fast method for applying the inverse of the matrix
\cite{Duff:1986:DMS:18753,Davis:2006:DMS:1196434}.
A sparse LU factorization, for instance, decomposes an invertible
sparse matrix $A \in \reals^{n \times n}$ into the product $A=LU$
of a lower triangular matrix $L \in \reals^{n \times n}$
and an upper triangular matrix $U \in \reals^{n \times n}$.
The relationship between $\nnz(A)$, $\nnz(L)$, and $\nnz(U)$
is complex and depends on the factorization algorithm
\cite[Chap.~6]{Davis:2006:DMS:1196434}.

We only discussed 1-D and 2-D DFTs and convolutions,
but these and related transforms can be extended to
arbitrarily many dimensions \cite{DM:84,9780898712858}.
Similarly, many wavelet transforms naturally operate on data indexed
by more than two dimensions \cite{KV:92,YDC:05,LD:07}.

\section{Compositions}\label{composition_sec}

In this section we consider compositions of FAOs.
In fact we have already discussed several linear functions that are
naturally and efficiently represented as compositions,
such as multiplication by a low-rank matrix and
sparse matrix mappings.
Here though we present a data structure and algorithm for efficiently
evaluating any composition and its adjoint,
which gives us an FAO representing the composition.

A composition of FAOs can be
represented using a directed acyclic graph (DAG)
with exactly one node with no incoming edges (the start node)
and exactly one node with no outgoing edges (the end node).
We call such a representation an \emph{FAO DAG}.



Each node in the FAO DAG stores the following attributes:
\begin{itemize}
\item An FAO $\Gamma = (f, \Phi_f, \Phi_{f^*})$.
Concretely, $f$ is a symbol identifying the function,
and $\Phi_f$ and $\Phi_{f^*}$ are executable code.
\item The data needed to evaluate $\Phi_f$ and $\Phi_{f^*}$.
\item A list $E_\mathrm{in}$ of incoming edges.
\item A list $E_\mathrm{out}$ of outgoing edges.
\end{itemize}
Each edge has an associated array.
The incoming edges to a node store the arguments to the node's FAO.
When the FAO is evaluated, it writes the result to the node's
outgoing edges.
Matrix arguments and outputs are stored in column-major order on the
edge arrays.

As an example, figure \ref{basic_dag} shows the FAO DAG for
the composition $f(x) = Ax + Bx$, where
$A \in \reals^{m \times n}$ and $B \in \reals^{m \times n}$
are dense matrices.
The $\FAOcopy$ node duplicates the input $x \in \reals^n$
into the multi-argument output $(x,x) \in \reals^n \times \reals^n$.
The $A$ and $B$ nodes multiply by $A$ and $B$, respectively.
The $\FAOsum$ node sums two vectors together.
The $\FAOcopy$ node is the start node,
and the $\FAOsum$ node is the end node.
The FAO DAG requires $O(mn)$ bytes to store,
since the $A$ and $B$ nodes store the matrices $A$ and $B$ and
their tranposes.
The edge arrays also require $O(mn)$ bytes of memory.

\begin{figure}
\centering
\begin {tikzpicture}[-latex ,auto ,node distance =1.5 cm and 2.5cm ,on grid ,
semithick ,
state/.style ={ rectangle ,top color =white , bottom color = white,
draw,black , text=black , minimum width =1 cm, minimum height = 0.75cm}]
\node[state] (C) {copy};
\node[state] (A) [above left=of C] {$A$};
\node[state] (B) [above right =of C] {$B$};
\node[state] (D) [above left =of B] {sum};
\path (C) edge node[below =0.15 cm] {} (A);
\path (C) edge node[below =0.15 cm] {} (B);
\path (A) edge  node[above] {} (D);
\path (B) edge  node[below =0.15 cm] {} (D);
\end{tikzpicture}
\caption{The FAO DAG for $f(x) = Ax + Bx$.}\label{basic_dag}
\end{figure}
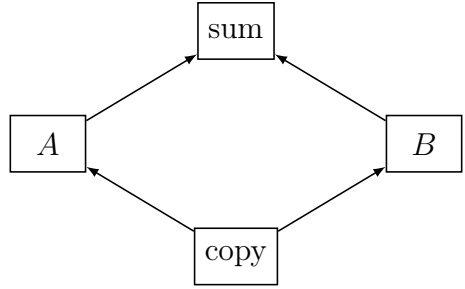

\subsection{Forward evaluation}
To evaluate the composition $f(x) = Ax + Bx$ using the FAO DAG in figure
\ref{basic_dag},
we first evaluate the start node on the input $x \in \reals^n$,
which copies $x$ onto both outgoing edges.
We evaluate the $A$ and $B$ nodes (serially or in parallel) on their
incoming edges, and write the results ($Ax$ and $Bx$) to
their outgoing edges.
Finally, we evaluate the end node on its incoming edges to obtain
the result $Ax + Bx$.

The general procedure for evaluating an FAO DAG is given in algorithm
\ref{fao_dag_eval_alg}.
The algorithm evaluates the nodes in a topological order.
The total flop count is the sum of the flops from evaluating the
algorithm $\Phi_f$ on each node.
If we allocate all scratch space needed by the FAO algorithms in advance,
then no memory is allocated during the algorithm.


\begin{algorithm}
\caption{Evaluate an FAO DAG.}\label{fao_dag_eval_alg}
\begin{algorithmic}
\Require{$G = (V, E)$ is an FAO DAG representing a function $f$.
$V$ is a list of nodes. $E$ is a list of edges.
$I$ is a list of inputs to $f$.
$O$ is a list of outputs from $f$.
Each element of $I$ and $O$ is represented as an array.}
\Statex
\State Create edges whose arrays are the elements of $I$
and save them as the list of incoming edges for the start node.
\State Create edges whose arrays are the elements of $O$
and save them as the list of outgoing edges for the end node.
\State Create an empty queue $Q$ for nodes that are ready to evaluate.
\State Create an empty set $S$ for nodes that have been evaluated.
\State Add $G$'s start node to $Q$.
\While {$Q$ is not empty}
    \State $u \gets \text{pop the front node of } Q$.
    \State Evaluate $u$'s algorithm $\Phi_f$ on $u$'s incoming edges,
    writing the result to $u$'s outgoing edges.
    \State Add $u$ to $S$.
    \For {each edge $e = (u,v)$ in $u$'s $E_\mathrm{out}$}
        \If {for all edges $(p,v)$ in $v$'s $E_\mathrm{in}$,
             $p$ is in $S$}
             \State Add $v$ to the end of $Q$.
        \EndIf
    \EndFor
\EndWhile
\Statex
\Ensure{$O$ contains the outputs of $f$ applied to inputs $I$.}
\end{algorithmic}
\end{algorithm}

\subsection{Adjoint evaluation}
Given an FAO DAG $G$ representing a function $f$,
we can easily generate an FAO DAG $G^*$ representing the adjoint $f^*$.
We modify each node in $G$,
replacing the node's FAO $(f, \Phi_f, \Phi_{f^*})$
with the FAO $(f^*, \Phi_{f^*}, \Phi_f)$
and swapping $E_\mathrm{in}$ and $E_\mathrm{out}$.
We also reverse the orientation of each edge in $G$.
We can apply algorithm \ref{fao_dag_eval_alg} to the resulting graph $G^*$ to evaluate $f^*$.
Figure \ref{adjoint_dag} shows the FAO DAG in figure \ref{basic_dag} transformed into an FAO DAG for the adjoint.

\begin{figure}
\centering
\begin {tikzpicture}[-latex ,auto ,node distance =1.5 cm and 2.5cm ,on grid ,
semithick ,
state/.style ={ rectangle ,top color =white , bottom color = white,
draw,black , text=black , minimum width =1 cm, minimum height = 0.75cm}]
\node[state] (C) {sum};
\node[state] (A) [above left=of C] {$A^T$};
\node[state] (B) [above right =of C] {$B^T$};
\node[state] (D) [above left =of B] {copy};
\path (A) edge node[below =0.15 cm] {} (C);
\path (B) edge node[below =0.15 cm] {} (C);
\path (D) edge  node[above] {} (A);
\path (D) edge  node[below =0.15 cm] {} (B);
\end{tikzpicture}
\caption{The FAO DAG for $f^*(u) = A^Tu + B^Tu$ obtained by transforming the FAO DAG in figure \ref{basic_dag}.}\label{adjoint_dag}
\end{figure}
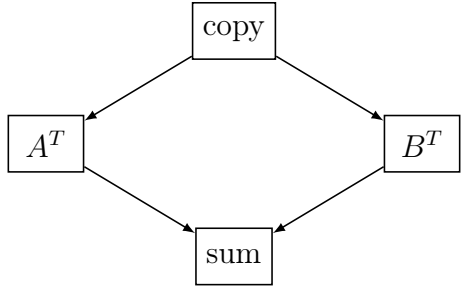

\subsection{Parallelism}
Algorithm \ref{fao_dag_eval_alg} can be easily parallelized,
since the nodes in the ready queue $Q$ can be evaluated in any order.
A simple parallel implementation could use a thread pool with $t$
threads to evaluate up to $t$ nodes in the ready queue at a time.
The extent to which parallelism speeds up evaluation of the composition
graph depends on how many parallel paths there are in the graph, \ie,
paths with no shared nodes.
The evaluation of individual nodes can also be parallelized by
replacing a node's algorithm $\Phi_f$ with a parallel variant.
For example, the standard algorithms for dense and sparse matrix
multiplication can be trivially parallelized.

\subsection{Optimizing the DAG}\label{opt_runtime}
The FAO DAG can often be transformed so that the output of
algorithm \ref{fao_dag_eval_alg} is the same but the algorithm is
executed more efficiently.
Such optimizations are especially important when the FAO DAG will
be evaluated on many different inputs
(as will be the case for matrix-free solvers, to be discussed later).
For example, the FAO DAG representing $f(x) = ABx + ACx$ where
$A,B,C \in \reals^{n \times n}$,
shown in figure \ref{ineff_dag},
can be transformed into the FAO DAG in figure \ref{eff_dag},
which requires one fewer multiplication by $A$.
The transformation is equivalent to rewriting $f(x) = ABx + ACx$ as
$f(x) = A(Bx + Cx)$.
Many other useful graph transformations can be derived from
the rewriting rules used in program analysis and code generation
\cite{Aho:2006:CPT:1177220}.

\begin{figure}
\centering
\begin {tikzpicture}[-latex ,auto ,node distance =1.5 cm and 2.5cm ,on grid ,
semithick ,
state/.style ={ rectangle ,top color =white , bottom color = white,
draw,black , text=black , minimum width =1 cm, minimum height = 0.75cm}]
\node[state] (C) {copy};
\node[state] (A) [above left=of C] {$B$};
\node[state] (B) [above right =of C] {$C$};
\node[state] (E) [above =of A] {$A$};
\node[state] (F) [above =of B] {$A$};
\node[state] (D) [above left =of F] {sum};
\path (C) edge node[below =0.15 cm] {} (A);
\path (C) edge node[below =0.15 cm] {} (B);
\path (A) edge node[below =0.15 cm] {} (E);
\path (B) edge node[below =0.15 cm] {} (F);
\path (E) edge  node[above] {} (D);
\path (F) edge  node[below =0.15 cm] {} (D);
\end{tikzpicture}
\caption{The FAO DAG for $f(x) = ABx + ACx$.}\label{ineff_dag}
\end{figure}
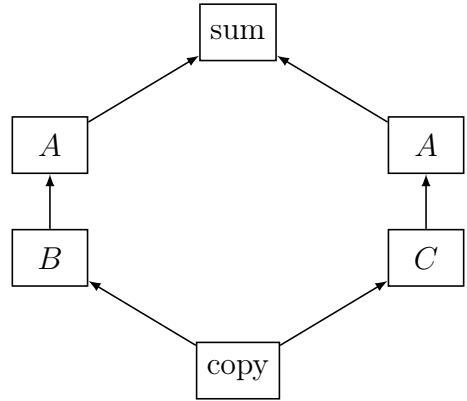

\begin{figure}
\centering
\begin {tikzpicture}[-latex ,auto ,node distance =1.5 cm and 2.5cm ,on grid ,
semithick ,
state/.style ={ rectangle ,top color =white , bottom color = white,
draw,black , text=black , minimum width =1 cm, minimum height = 0.75cm}]
\node[state] (C) {copy};
\node[state] (A) [above left=of C] {$B$};
\node[state] (B) [above right =of C] {$C$};
\node[state] (D) [above left =of B] {sum};
\node[state] (E) [above =of D] {$A$};
\path (C) edge node[below =0.15 cm] {} (A);
\path (C) edge node[below =0.15 cm] {} (B);
\path (A) edge node[above] {} (D);
\path (B) edge node[below =0.15 cm] {} (D);
\path (D) edge node[below =0.15 cm] {} (E);
\end{tikzpicture}
\caption{The FAO DAG for $f(x) = A(Bx + Cx)$.}\label{eff_dag}
\end{figure}
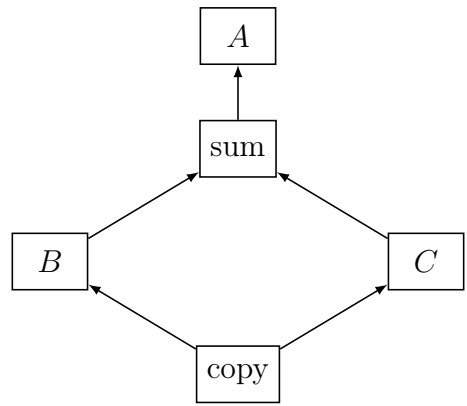

Sometimes graph transformations will involve pre-computation.
For example, if two nodes representing the composition
$f(x) = b^Tcx$, where $b,c \in \reals^n$, appear in an FAO DAG,
the DAG can be made more efficient by evaluating $\alpha = b^Tc$
and replacing the two nodes
with a single node for scalar multiplication by $\alpha$.

The optimal rewriting of a DAG will depend on the hardware and overall
architecture on which the multiplication algorithm is being run.
For example, if the algorithm is being run on a distributed computing
cluster then a node representing multiplication by a large matrix
\[
A = \left[\begin{array}{cc}
A_{11} & A_{12} \\
A_{21} & A_{22}
\end{array}\right],
\]
could be split into separate nodes for each block, with the nodes stored
 on different computers.
This rewriting would be necessary if the matrix $A$ is so large it
cannot be stored on a single machine.
The literature on optimizing compilers suggests many approaches
to optimizing an FAO DAG for evaluation on a particular architecture
\cite{Aho:2006:CPT:1177220}.

\subsection{Reducing the memory footprint}\label{opt_memory}
In a naive implementation, the total bytes needed to represent an FAO
DAG $G$, with node set $V$ and edge set $E$, is the sum of the bytes of data on each node $u \in V$
and the bytes of memory needed for the array on each edge $e \in E$.
A more sophisticated approach can substantially reduce the
memory needed.
For example, when the same FAO occurs more than once in $V$,
duplicate nodes can share data.

We can also reuse memory across edge arrays.
The key is determining which arrays can never be in use
at the same time during algorithm \ref{fao_dag_eval_alg}.
An array for an edge $(u,v)$ is in use if node
$u$ has been evaluated but node $v$ has not been evaluated.
The arrays for edges $(u_1,v_1)$ and $(u_2,v_2)$ can never
be in use at the same time if and only if there is a directed path from
$v_1$ to $u_2$ or from $v_2$ to $u_1$.
If the sequence in which the nodes will be evaluated is fixed,
rather than following an unknown topological ordering,
then we can say precisely which arrays will be in use at the same time.

The next step is to map the edge arrays onto a global array,
keeping the global array as small as possible.
Let $L(e)$ denote the length of edge $e$'s array and
$U \subseteq E \times E$ denote the set of pairs of edges
whose arrays may be in use at the same time.
Formally, we want to solve the optimization problem
\begin{equation}\label{edge_array_prob}
\begin{array}{ll}
\mbox{minimize}   & \max_{e \in E}\{z_e + L(e)\} \\
\mbox{subject to} & [z_{e},z_{e} + L(e) - 1]\cap
[z_{f},z_{f} + L(f) - 1] = \emptyset, \quad (e,f) \in U \\
& z_e \in \{1,2,\ldots \}, \quad e \in E,
\end{array}
\end{equation}
where the $z_e$ are the optimization variables and represent
the index in the global array where edge $e$'s array begins.

When all the edge arrays are the same length, problem
(\ref{edge_array_prob}) is equivalent to finding the chromatic number
of the graph with vertices $E$ and edges $U$.
Problem (\ref{edge_array_prob}) is thus NP-hard in general \cite{Kar:72}.
A reasonable heuristic for problem (\ref{edge_array_prob}) is to first
find a graph coloring of $(E,U)$ using one of the many efficient
algorithms for finding graph colorings that use a small number of
colors; see, \eg, \cite{Hal:93,Bre:79}.
We then have a mapping $\phi$ from colors to sets of edges assigned
to the same color.
We order the colors arbitrarily as $c_1,\ldots,c_k$ and assign the
$z_e$ as follows:
\[
z_e = \begin{cases}
1, \quad e \in \phi(c_1) \\
\max_{f \in \phi(c_{i-1})}\{z_f + L(f) \},
\quad e \in \phi(c_i), \quad i > 1.
\end{cases}
\]

Additional optimizations can be made based on the unique characteristics
of different FAOs. For example, the outgoing edges from a $\FAOcopy$
node can share the incoming edge's array until the outgoing edges'
arrays are written to (\ie, copy-on-write).
Another example is that the outgoing edges from a $\FAOsplit$ node
can point to segments of the array on the incoming edge.
Similarly, the incoming edges on a $\vstack$ node can point to
segments of the array on the outgoing edge.


\subsection{Software implementations}
Several software packages have been developed
for constructing and evaluating compositions of linear functions.
The MATLAB toolbox SPOT allows users to construct expressions involving
both fast transforms, like convolution and the DFT,
and standard matrix multiplication \cite{SPOT}.
TFOCS, a framework in MATLAB for solving convex problems using a variety
of first-order algorithms,
provides functionality for constructing and composing FAOs \cite{BCG:11}.
The Python package \texttt{linop} provides methods for constructing FAOs
and combining them into linear expressions \cite{linop}.
Halide is a domain specific language for image processing
that makes it easy to optimize compositions of fast transforms for a
variety of architectures \cite{Halide}.

\section{Cone programs and solvers}\label{cone_prog_and_solvers_sec}

\subsection{Cone programs}\label{cone_prog_sec}

A cone program is a convex optimization problem of the form
\begin{equation}\label{cone_prog}
\begin{array}{ll}
\mbox{minimize}   & c^{T}x \\
\mbox{subject to} & Ax + b  \in \K,
\end{array}
\end{equation}
where $x \in \reals^n$ is the optimization variable,
$\K$ is a convex cone,
and $A \in \reals^{m \times n}$,
$c \in \reals^n$,
and $b \in \reals^m$ are problem data.
Cone programs are a broad class that include linear programs, second-order cone programs, and semidefinite programs as special cases
\cite{NesNem:92,BoV:04}.
We call the cone program \emph{matrix-free} if $A$ is
represented implicitly as an FAO,
rather than explicitly as a dense or sparse matrix.

The convex cone $\K$ is typically a Cartesian product of simple convex
cones from the following list:
\begin{itemize}
\item Zero cone: $\K_{0} = \{0\}$.
\item Free cone: $\K_{\mathrm{free}} = \reals$.
\item Nonnegative cone:
$\K_{+} = \{x \in \reals \mid x \geq 0 \}$.
\item Second-order cone:
$\K_{\mathrm{soc}} =
\{ (x,t) \in \reals^{n+1} \mid x \in \reals^n,
\; t \in \reals, \;  \|x\|_2 \leq t \}$.
\item Positive semidefinite cone:
$\K_{\mathrm{psd}} = \{\vect(X) \mid X \in \symm^n, \; z^TXz \geq 0
\mbox{ for all } z \in \reals^n \}$.
\item Exponential cone (\cite[\S 6.3.4]{parikh2013proximal}):
\[
\K_{\mathrm{exp}} =
\{(x, y, z) \in \reals^3
\mid y > 0, \; y e^{x/y} \leq z\}
\cup \{(x, y, z) \in \reals^3
\mid x \leq 0, \; y = 0, \; z \geq 0\}.
\]
\item Power cone (\cite{nes:06,sy:14,KH:14}):
\[
\K_{\mathrm{pwr}}^a =
\{(x, y, z) \in \reals^3 \mid x^ay^{(1-a)} \geq |z|, \; x \geq 0, \; y \geq 0 \},
\]
where $a \in [0,1]$.
\end{itemize}
These cones are useful in expressing common problems
(via canonicalization), and can be handled by various solvers (as discussed
below).
Note that all the cones are subsets of $\reals^n$,
\ie, real vectors.
It might be more natural to view the elements of a cone as matrices or tuples,
but viewing the elements as vectors simplifies the matrix-free  canonicalization algorithm in \S\ref{canon_sec}.

Cone programs that include only cones from certain subsets of the list
above have special names.
For example, if the only cones are zero, free,
and nonnegative cones, the cone program is a linear program;
if in addition it includes the second-order cone,
it is called a second-order cone program.
A well studied special case is so-called symmetric cone programs,
which include the zero, free, nonnegative, second-order,
and positive semidefinite cones.
Semidefinite programs, where the cone constraint consists of a single
positive semidefinite cone, are another common case.

\subsection{Cone solvers}\label{cone_solvers_sec}

Many methods have been developed to solve cone programs,
the most widely used being interior-point methods; see, \eg,
\cite{BoV:04,NN:94,NW:06,Wright:87,Ye:11}.

\paragraph{Interior-point.}
A large number of interior-point cone solvers have been implemented.
Most support symmetric cone programs.
SDPT3 \cite{toh1999sdpt3} and SeDuMi \cite{Sturm1999}
are open-source solvers implemented in MATLAB;
CVXOPT \cite{CVXOPT}
is an open-source solver implemented in Python;
MOSEK \cite{mosek}
is a commercial solver with interfaces to many languages.
ECOS is an open-source cone solver written in library-free C that
supports second-order cone programs
\cite{bib:Domahidi2013ecos};
Akle extended ECOS to support the exponential cone \cite{Akle15}.
DSDP5 \cite{dsdp5} and SDPA \cite{SDPA} are open-source solvers
for semidefinite programs implemented in C and C++, respectively.

\paragraph{First-order.}
First-order methods are an alternative to interior-point methods that
scale more easily to large cone programs,
at the cost of lower accuracy.
PDOS \cite{CDPB:13} is a first-order cone solver based on the
alternating direction method of multipliers (ADMM) \cite{BP:11}.
PDOS supports second-order cone programs.
POGS \cite{fougner2015parameter} is an ADMM based solver that runs
on a GPU, with a version that is similar to PDOS and
targets second-order cone programs.
SCS is another ADMM-based cone solver,
which supports symmetric cone programs as well as
the exponential and power cones \cite{SCSpaper}.
Many other first-order algorithms can be applied to cone programs
(\eg, \cite{LLM:11,CP:11,pock2011diagonal}),
but none have been implemented as a robust, general purpose cone solver.

\paragraph{Matrix-free.}
Matrix-free cone solvers are an area of active research,
and a small number have been developed.
PENNON is a matrix-free semidefinite program (SDP) solver \cite{KS:09}.
PENNON solves a series of unconstrained optimization problems using
Newton's method.
The Newton step is computed using a preconditioned conjugate gradient
method, rather than by factoring the Hessian directly.
Many other matrix-free algorithms for solving SDPs have been proposed
(\eg, \cite{choi2000solving,FKM:02,T:04,ZST:10}).

Several matrix-free solvers have been developed for quadratic programs
(QPs), which are a superset of linear programs and a
subset of second-order cone programs.
Gondzio developed a matrix-free interior-point method
for QPs that solves linear systems using a
preconditioned conjugate gradient method
\cite{GondzioMatFree:12,Gondzio:12,HeS:52}.
PDCO is a matrix-free interior-point solver that can solve QPs
\cite{PDCO}, using LSMR to solve linear systems \cite{fong2011lsmr}.

\section{Matrix-free canonicalization}\label{canon_sec}


\subsection{Canonicalization}

Canonicalization is an algorithm that takes as input a
data structure representing a general convex optimization
problem and outputs a data structure representing an equivalent
cone program.
By solving the cone program, we recover the solution
to the original optimization problem.
This approach is used by convex optimization modeling systems such as
YALMIP \cite{Lofberg:04}, CVX \cite{cvx}, CVXPY \cite{cvxpy_paper},
and Convex.jl \cite{cvxjl}.
The same technique is used in the code generators CVXGEN \cite{MB:12}
and QCML \cite{QCML}.

The downside of canonicalization's generality is that special structure
in the original problem may be lost during the transformation into a
cone program.
In particular, current methods of canonicalization convert
fast linear transforms in the original
problem into multiplication by a dense or sparse matrix,
which makes the final cone program far more costly to solve than
the original problem.

The canonicalization algorithm can be modified, however, so that fast
linear transforms are preserved.
The key is to represent all linear functions arising during the
canonicalization process as FAO DAGs instead of as sparse matrices.
The FAO DAG representation of the final cone program can be used by
a matrix-free cone solver to solve the cone program.
The modified canonicalization algorithm never forms explicit matrix
representations of linear functions.
Hence we call the algorithm \emph{matrix-free canonicalization}.

The remainder of this section has the following outline:
In \S\ref{informal_overview_sec} we give an informal
overview of the matrix-free canonicalization algorithm.
In \S\ref{expr_tree_sec} we define the expression DAG data structure,
which is used throughout the matrix-free canonicalization algorithm.
In \S\ref{cvx_prog_sec} we define the data structure used to represent
convex optimization problems as input to the algorithm.
In \S\ref{cone_prog_repr_sec}
we define the representation of a cone program
output by the matrix-free canonicalization algorithm.
In \S\ref{algorithm_sec} we present the matrix-free canonicalization
algorithm itself.

For clarity, we move some details of canonicalization to the appendix.
In \S\ref{equivalence_sec} we give a precise definition of the
equivalence between the cone program output by
the canonicalization algorithm and the original convex optimization
problem given as input.
In \S\ref{sparse_mat_sec} we explain how the standard canonicalization
algorithm generates a sparse matrix representation of a cone program.

\subsection{Informal overview}\label{informal_overview_sec}
In this section we give an informal overview of the matrix-free
canonicalization algorithm.
Later sections define the data structures used in the algorithm
and make the procedure described in this section formal and explicit.

We are given an optimization problem
\begin{equation}\label{general_opt_prob}
\begin{array}{ll}
\mbox{minimize} & f_0(x) \\
\mbox{subject to} & f_i(x) \leq 0, \quad i=1,\ldots,p \\
& h_i(x) + d_i = 0, \quad i=1,\ldots,q,
\end{array}
\end{equation}
where $x \in \reals^n$ is the optimization variable,
$f_0:\reals^n \to \reals,\ldots,f_p:\reals^n \to \reals$ are convex
functions, $h_1:\reals^n \to \reals^{m_1},\ldots,
h_q:\reals^n \to \reals^{m_q}$ are linear functions,
and $d_1 \in \reals^{m_1}, \ldots, d_q \in \reals^{m_q}$ are
vector constants.
Our goal is to convert the problem into an equivalent matrix-free
cone program, so that we can solve it using a matrix-free cone solver.

We assume that the problem satisfies a set of requirements known as
\emph{disciplined convex programming} \cite{Grant2004,GBY:06}.
The requirements ensure that each of the $f_0,\ldots,f_p$ can be
represented as partial minimization over a cone program.
Let each function $f_i$ have the cone program representation
\[
\begin{array}{lcll}
f_i(x) &= &\mbox{minimize} ~(\text{over } t^{(i)}) &
g^{(i)}_0(x,t^{(i)}) + e^{(i)}_0 \\
& &\mbox{subject to} & g^{(i)}_j(x,t^{(i)}) + e^{(i)}_j \in \K^{(i)}_j,
\quad j=1,\ldots,r^{(i)}, \\
\end{array}
\]
where $t^{(i)} \in \reals^{s^{(i)}}$ is the optimization variable,
$g^{(i)}_0,\ldots,g^{(i)}_{r^{(i)}}$
are linear functions,
$e^{(i)}_0,\ldots,e^{(i)}_{r^{(i)}}$ are vector constants,
and $\K^{(i)}_1,\ldots,\K^{(i)}_{r^{(i)}}$ are convex cones.

We rewrite problem (\ref{general_opt_prob}) as the equivalent
cone program
\begin{equation}\label{general_cone_prog}
\begin{array}{ll}
\mbox{minimize} & g^{(0)}_0(x, t^{(0)}) + e^{(0)}_0  \\
\mbox{subject to} & -g^{(i)}_0(x, t^{(i)}) - e^{(i)}_0 \in \K_{+},
\quad i=1,\ldots,p, \\
& g^{(i)}_j(x, t^{(i)}) + e^{(i)}_j \in \K^{(i)}_j
\quad i=1,\ldots,p, \quad  j=1,\ldots,r^{(i)} \\
& h_i(x) + d_i \in \K_{0}^{m_i}, \quad i=1,\ldots,q.
\end{array}
\end{equation}
We convert problem (\ref{general_cone_prog}) into the standard form
for a matrix-free cone program given in (\ref{cone_prog}) by
representing $g^{(0)}_0$ as the inner product with a vector
$c \in \reals^{n + s^{(0)}}$,
concatenating the $d_i$ and $e^{(i)}_j$ vectors into a
single vector $b$,
and representing the matrix $A$ implicitly as the linear function
that stacks the outputs of all the $h_i$ and $g^{(i)}_j$
(excluding the objective $g^{(0)}_0$) into a single vector.

\subsection{Expression DAGs}\label{expr_tree_sec}

The canonicalization algorithm uses a data structure called an
\emph{expression DAG} to represent functions in an optimization problem.
Like the FAO DAG defined in \S\ref{composition_sec},
an expression DAG encodes a composition of functions as a DAG
where a node represents a function
and an edge from a node $u$ to a node $v$ signifies that an output of
$u$ is an input to $v$.
Figure \ref{basic_expr_dag} shows an expression DAG for the composition
$f(x) = \|Ax\|_2 + 3$, where $x \in \reals^n$ and
$A \in \reals^{m \times n}$.

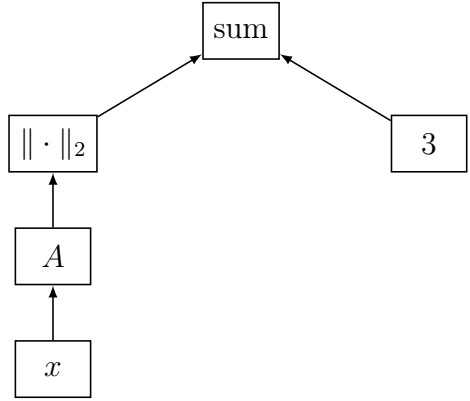
\begin{figure}
\centering
\begin {tikzpicture}[-latex ,auto ,node distance =1.5 cm and 2.5 cm,
on grid ,semithick ,
state/.style ={ rectangle ,top color =white , bottom color = white,
draw,black , text=black , minimum width =1 cm, minimum height = 0.75cm}]
\node[state] (A) {$x$};
\node[state] (B) [above =of A] {$A$};
\node[state] (C) [above =of B] {$\|\cdot\|_2$};
\node[state] (D) [above right=of C] {sum};
\node[state] (E) [below right =of D] {$3$};
\path (A) edge node[below =0.15 cm] {} (B);
\path (B) edge node[below =0.15 cm] {} (C);
\path (C) edge node[below =0.15 cm] {} (D);
\path (E) edge node[below =0.15 cm] {} (D);
\end{tikzpicture}
\caption{The expression DAG for $f(x) = \|Ax\|_2 + 3$.}\label{basic_expr_dag}
\end{figure}

Formally, an expression DAG is a connected DAG with one node with
no outgoing edges (the end node) and one or more nodes with no
incoming edges (start nodes).
Each node in an expression DAG has the following attributes:
\begin{itemize}
\item A symbol representing a function $f$.
\item The data needed to parameterize the function,
such as the power $p$ for the function $f(x) = x^p$.
\item A list $E_\mathrm{in}$ of incoming edges.
\item A list $E_\mathrm{out}$ of outgoing edges.
\end{itemize}

Each start node in an expression DAG is either a constant function or
a variable.
A variable is a symbol that labels a node input.
If two nodes $u$ and $v$ both have incoming edges from variable nodes
with symbol $t$, then the inputs to $u$ and $v$ are the same.

We say an expression DAG is affine if every non-start node represents
a linear function.
If in addition every start node is a variable, we say the expression
DAG is linear.
We say an expression DAG is constant if it contains no variables, \ie,
every start node is a constant.

\subsection{Optimization problem representation}\label{cvx_prog_sec}

An \emph{optimization problem representation} (OPR) is a data structure
that represents a convex optimization problem.
The input to the matrix-free canonicalization algorithm is an OPR.
An OPR can encode any mathematical optimization problem of the form
\begin{equation}\label{opr_prob}
\begin{array}{ll} \mbox{minimize}~(\text{over } y~\mathrm{w.r.t.}~\K_0)
& f_0(x,y)\\
\mbox{subject to} & f_i(x,y) \in \K_i, \quad i=1, \ldots, \ell,
\end{array}
\end{equation}
where $x \in \reals^n$ and $y \in \reals^m$ are the optimization
variables,
$\K_0$ is a proper cone,
$\K_1,\ldots,\K_\ell$ are convex cones,
and for $i=0,\ldots,\ell$,
we have $f_i : \reals^n \times \reals^m \to \reals^{m_i}$ where
$\K_i \subseteq \reals^{m_i}$.
(For background on convex optimization with respect to a cone, see,
\eg, \cite[\S 4.7]{BoV:04}.)

Problem (\ref{opr_prob}) is more complicated than the standard
definition of a convex optimization problem given in
(\ref{general_opt_prob}).
The additional complexity is necessary so that OPRs can encode partial
minimization over cone programs,
which can involve minimization with respect to a cone and constraints
other than equalities and inequalities.
These partial minimization problems play a major role in the
canonicalization algorithm.
Note that we can easily represent equality and inequality
constraints using the zero and nonnegative cones.

Concretely, an OPR is a tuple $(s,o,C)$ where
\begin{itemize}
 \item The element $s$ is a tuple $(V,\K)$ representing the problem's
 objective sense.
 The element $V$ is a set of symbols encoding the variables being
 minimized over.
 The element $\K$ is a symbol encoding the proper cone the problem
 objective is being minimized with respect to.
 \item The element $o$ is an expression DAG representing the problem's
 objective function.
 \item The element $C$ is a set representing the problem's constraints.
 Each element $c_i \in C$ is a tuple $(e_i, \K_i)$ representing a
 constraint of the form $f(x,y) \in \K$.
 The element $e_i$ is an expression DAG representing the function $f$
 and $\K_i$ is a symbol encoding the convex cone $\K$.
\end{itemize}

The matrix-free canonicalization algorithm can only operate on OPRs
that satisfy the two DCP requirements \cite{Grant2004,GBY:06}.
The first requirement is that each nonlinear function in the OPR
have a known representation as partial minimization over a cone program.
See \cite{GB:08} for many examples of such representations.

The second requirement is that the objective $o$ be verifiable as
convex with respect to the cone $\K$ in the objective sense $s$ by the
DCP composition rule.
Similarly, for each element $(e_i, \K_i) \in C$,
the constraint that the function represented by $e_i$ lie in the convex
cone represented by $\K_i$
must be verifiable as convex by the composition rule.
The DCP composition rule determines the curvature
of a composition $f(g_1(x),\ldots,g_k(x))$ from the curvatures and
ranges of the arguments $g_1,\ldots,g_k$,
the curvature of the function $f$,
and the monotonicity of $f$ on the range of its arguments.
See \cite{Grant2004} and \cite{cvxjl} for a full discussion of the DCP
composition rule.
Additional rules are used to determine the range of a composition
from the range of its arguments.


Note that it is not enough for the objective and constraints to be
convex.
They must also be structured so that the DCP composition rule can verify
their convexity.
Otherwise the cone program output by the matrix-free canonicalization
algorithm is not guaranteed to be equivalent to the original problem.

To simplify the exposition of the canonicalization algorithm,
we will also require that the objective sense $s$ represent
minimization over all the variables in the problem with respect to the
nonnegative cone, \ie,
the standard definition of minimization.
The most general implementation of canonicalization would also accept
OPRs that can be transformed into an equivalent OPR with an objective
sense that meets this requirement.

\subsection{Cone program representation}\label{cone_prog_repr_sec}

The matrix-free canonicalization algorithm outputs a tuple
$(c_{\mathrm{arr}},d_{\mathrm{arr}},b_{\mathrm{arr}},G,
\K_{\mathrm{list}})$ where
\begin{itemize}
\item The element $c_{\mathrm{arr}}$ is a length $n$ array representing
a vector $c \in \reals^n$.
\item The element $d_{\mathrm{arr}}$ is a length one array representing
a scalar $d \in \reals$.
\item The element $b_{\mathrm{arr}}$ is a length $m$ array representing
a vector $b \in \reals^m$.
\item The element $G$ is an FAO DAG representing a linear function
$f(x) = Ax$, where $A \in \reals^{m \times n}$.
\item The element $\K_{\mathrm{list}}$ is a list of symbols representing
the convex cones $(\K_1, \ldots, \K_\ell)$ .
\end{itemize}
The tuple represents the matrix-free cone program
\begin{equation}\label{cone_repr_prog}
\begin{array}{ll}
\mbox{minimize}   & c^{T}x + d \\
\mbox{subject to} & Ax + b  \in \K,
\end{array}
\end{equation}
where $\K = \K_1 \times \cdots \times \K_\ell$.

We can use the FAO DAG $G$ and algorithm
\ref{fao_dag_eval_alg} to represent $A$ as an FAO,
\ie, export methods for multiplying by $A$ and $A^T$.
These two methods are all a matrix-free
cone solver needs to efficiently solve problem (\ref{cone_repr_prog}).

\subsection{Algorithm}\label{algorithm_sec}

The matrix-free canonicalization algorithm can be broken down into
subroutines.
We describe these subroutines before presenting the overall algorithm.

\paragraph{\texttt{Conic-Form}.}
The \texttt{Conic-Form} subroutine takes an OPR as input and returns
an equivalent OPR where every non-start node in the objective and
constraint expression DAGs represents a linear function.
The output of the \texttt{Conic-Form} subroutine represents a
cone program,
but the output must still be transformed into a data structure that a
cone solver can use,
\eg~the cone program representation described in
\S\ref{cone_prog_repr_sec}.

The general idea of the \texttt{Conic-Form} algorithm is to replace
each nonlinear function in the OPR with an OPR representing partial
minimization over a cone program.
Recall that the canonicalization algorithm requires that all
nonlinear functions in the problem be representable as partial
minimization over a cone program.
The OPR for each nonlinear function is spliced into the full OPR.
We refer the reader to \cite{GB:08} and \cite{cvxjl} for a full
discussion of the \texttt{Conic-Form} algorithm.


The \texttt{Conic-Form} subroutine preserves fast linear transforms
in the problem.
All linear functions in the original OPR are present in the OPR output
by \texttt{Conic-Form}.
The only linear functions added are ones like $\FAOsum$ and scalar
multiplication that are very efficient to evaluate.
Thus, evaluating the FAO DAG representing the final cone program will be
as efficient as evaluating all the linear functions in the original
problem.

\paragraph{\texttt{Linear} and \texttt{Constant}.}
The \texttt{Linear} and \texttt{Constant} subroutines take an affine
expression DAG as input and return the DAG's linear and constant
components, respectively.
Concretely, the \texttt{Linear} subroutine returns a copy of the
input DAG where every constant start node is replaced with a variable
start node and a node mapping the variable output to a vector
(or matrix) of zeros with the same dimensions as the constant.
The \texttt{Constant} subroutine returns a copy of the input DAG
where every variable start node is replaced with a zero-valued constant
node of the same dimensions.
Figures \ref{linear_dag} and \ref{constant_dag} show the results of
applying the \texttt{Linear} and
\texttt{Constant} subroutines to an expression DAG representing
$f(x) = x + 2$, as depicted in figure \ref{initial_affine_dag}.

\begin{figure}
\centering
\begin {tikzpicture}[-latex ,auto ,node distance =1.5 cm and 2.5 cm,
on grid ,semithick ,
state/.style ={ rectangle ,top color =white , bottom color = white,
draw,black , text=black , minimum width =1 cm, minimum height = 0.75cm}]
\node[state] (A) {$x$};
\node[state] (C) [above right=of A] {sum};
\node[state] (B) [below right=of C] {$2$};
\path (A) edge node[below =0.15 cm] {} (C);
\path (B) edge node[below =0.15 cm] {} (C);
\end{tikzpicture}
\caption{The expression DAG for $f(x) = x + 2$.}\label{initial_affine_dag}
\end{figure}
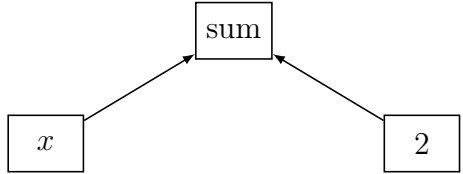

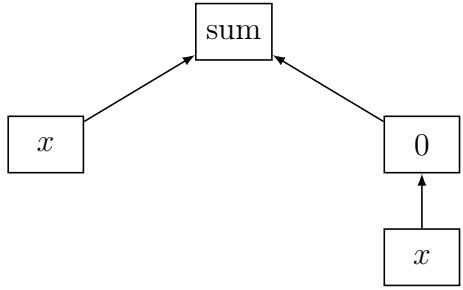
\begin{figure}
\centering
\begin {tikzpicture}[-latex ,auto ,node distance =1.5 cm and 2.5 cm,
on grid ,semithick ,
state/.style ={ rectangle ,top color =white , bottom color = white,
draw,black , text=black , minimum width =1 cm, minimum height = 0.75cm}]
\node[state] (A) {$x$};
\node[state] (C) [above right=of A] {sum};
\node[state] (B) [below right=of C] {$0$};
\node[state] (D) [below =of B] {$x$};
\path (A) edge node[below =0.15 cm] {} (C);
\path (B) edge node[below =0.15 cm] {} (C);
\path (D) edge node[below =0.15 cm] {} (B);
\end{tikzpicture}
\caption{The \texttt{Linear} subroutine applied to the expression DAG in
figure \ref{initial_affine_dag}.}\label{linear_dag}
\end{figure}

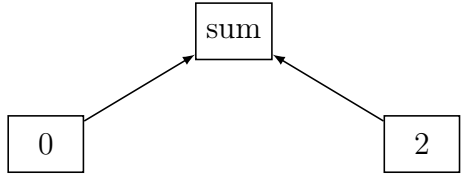
\begin{figure}
\centering
\begin {tikzpicture}[-latex ,auto ,node distance =1.5 cm and 2.5 cm,
on grid ,semithick ,
state/.style ={ rectangle ,top color =white , bottom color = white,
draw,black , text=black , minimum width =1 cm, minimum height = 0.75cm}]
\node[state] (A) {$0$};
\node[state] (C) [above right=of A] {sum};
\node[state] (B) [below right=of C] {$2$};
\path (A) edge node[below =0.15 cm] {} (C);
\path (B) edge node[below =0.15 cm] {} (C);
\end{tikzpicture}
\caption{The \texttt{Constant} subroutine applied to the expression DAG
in figure \ref{initial_affine_dag}.}\label{constant_dag}
\end{figure}

\paragraph{\texttt{Evaluate}.}
The \texttt{Evaluate} subroutine takes a constant expression DAG as
input and returns an array.
The array contains the value of the function represented by the
expression DAG.
If the DAG evaluates to a matrix $A \in \reals^{m \times n}$,
the array represents $\vect(A)$.
Similarly, if the DAG evaluates to multiple output vectors
$(b_1, \ldots, b_k) \in \reals^{n_1} \times \cdots \times \reals^{n_k}$,
the array represents $\vstack(b_1, \ldots, b_k)$.
For example, the output of the \texttt{Evaluate} subroutine on the
expression DAG in figure \ref{constant_dag} is a length one array with
first entry equal to $2$.


\paragraph{\texttt{Graph-Repr}.}
The \texttt{Graph-Repr} subroutine takes a list of linear expression
DAGs, $(e_1,\ldots,e_\ell)$, and an ordering over the variables in the
expression DAGs, $<_{V}$,
as input and outputs an FAO DAG $G$.
We require that the end node of each expression DAG represent a
function with a single vector as output.

We construct the FAO DAG $G$ in three steps.
In the first step, we combine the expression DAGs into a single
expression DAG $H^{(1)}$ by creating a $\vstack$ node
and adding an edge from the end node of each expression DAG to
the new node.
The expression DAG $H^{(1)}$ is shown in figure \ref{combo_dag}.

\begin{figure}
\centering
\begin {tikzpicture}[-latex ,auto ,node distance =1.5 cm and 2.5cm ,
on grid ,semithick ,
state/.style ={ rectangle ,top color =white , bottom color = white,
draw,black , text=black , minimum width =1 cm, minimum height = 0.75cm}]
\node[state] (A) {$e_1$};
\node (B) [right = of A] {\mydots};
\node[state] (C) [right = of B] {$e_\ell$};
\node[state] (D) [above =of B] {vstack};
\path (A) edge node[below =0.15 cm] {} (D);
\path (C) edge node[below =0.15 cm] {} (D);
\end{tikzpicture}
\caption{The expression DAG for $\vstack(e_1, \ldots, e_\ell)$.}\label{combo_dag}
\end{figure}
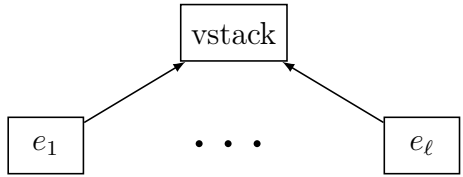

In the second step, we transform $H^{(1)}$ into an expression DAG
$H^{(2)}$ with a single start node.
Let $x_1,\ldots,x_k$ be the variables in $(e_1,\ldots,e_\ell)$ ordered
by $<_{V}$.
Let $n_i$ be the length of $x_i$ if the variable is a vector and of
$\vect(x_i)$ if the variable is a matrix, for $i=1,\ldots,k$.
We create a start node representing the function
$\FAOsplit: \reals^n \to \reals^{n_1} \times \cdots \times \reals^{n_k}$.
For each variable $x_i$,
we add an edge from output $i$ of the start node to a $\FAOcopy$ node
and edges from that $\FAOcopy$ node to all the nodes representing $x_i$.
If $x_i$ is a vector, we replace all the nodes representing $x_i$ with
nodes representing the identity function.
If $x_i$ is a matrix, we replace all the nodes representing $x_i$ with
$\mat$ nodes.
The transformation from $H^{(1)}$ to $H^{(2)}$ when $\ell=1$ and $e_1$
represents $f(x) = x + A(x+y)$,
where $x,y \in \reals^n$ and $A \in \reals^{n \times n}$,
are depicted in figures \ref{pre_transform_dag} and
\ref{post_transform_dag}.

In the third and final step, we transform $H^{(2)}$ from an expression
DAG into an FAO DAG $G$.
$H^{(2)}$ is almost an FAO DAG, since each node represents a linear
function and the DAG has a single start and end node.
To obtain $G$ we simply add the node and edge attributes needed in an
FAO DAG.
For each node $u$ in $H^{(2)}$ representing the function $f$,
we add to $u$ an FAO $(f, \Phi_f, \Phi_{f^*})$
and the data needed to evaluate $\Phi_f$ and $\Phi_{f^*}$.
The node already has the required lists of incoming and outgoing edges.
We also add an array to each of $H^{(2)}$'s edges.

\begin{figure}
\centering
\begin {tikzpicture}[-latex ,auto ,node distance =1.5 cm and 2.5 cm,
on grid ,semithick ,
state/.style ={ rectangle ,top color =white , bottom color = white,
draw,black , text=black , minimum width =1 cm, minimum height = 0.75cm}]
\node[state] (A) {$x$};
\node[state] (C) [above right=of A] {sum};
\node[state] (B) [below right=of C] {$A$};
\node[state] (G) [below =of B] {sum};
\node[state] (F) [below left=of G] {$x$};
\node[state] (D) [below right=of G] {$y$};
\node[state] (E) [above =of C] {vstack};
\path (A) edge node[below =0.15 cm] {} (C);
\path (B) edge node[below =0.15 cm] {} (C);
\path (C) edge node[below =0.15 cm] {} (E);
\path (D) edge node[below =0.15 cm] {} (G);
\path (F) edge node[below =0.15 cm] {} (G);
\path (G) edge node[below =0.15 cm] {} (B);
\end{tikzpicture}
\caption{The expression DAG $H^{(1)}$ when $\ell=1$ and $e_1$ represents
$f(x,y) = x + A(x + y)$.}\label{pre_transform_dag}
\end{figure}
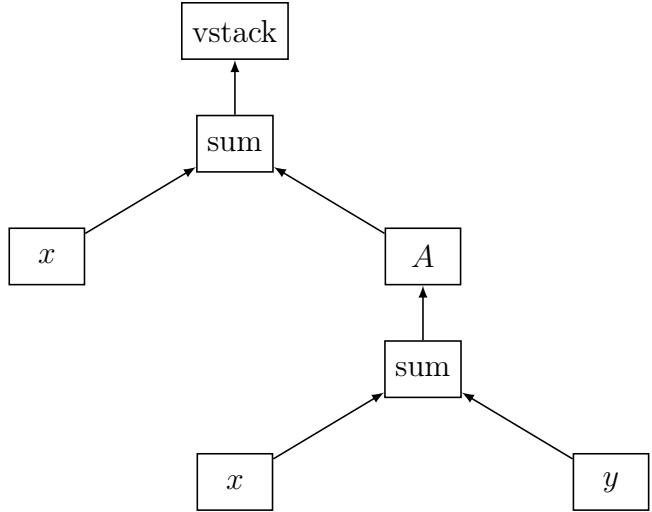

\begin{figure}
\centering
\begin {tikzpicture}[-latex ,auto ,node distance =1.5 cm and 2.5 cm,
on grid ,semithick ,
state/.style ={ rectangle ,top color =white , bottom color = white,
draw,black , text=black , minimum width =1 cm, minimum height = 0.75cm}]
\node[state] (A) {$I$};
\node[state] (C) [above right=of A] {sum};
\node[state] (B) [below right=of C] {$A$};
\node[state] (G) [below =of B] {sum};
\node[state] (F) [below left=of G] {$I$};
\node[state] (D) [below right=of G] {$I$};
\node[state] (E) [above =of C] {vstack};
\node[state] (H) [below left=of F] {copy};
\node[state] (I) [below =of D] {copy};
\node[state] (J) [below right=of H] {split};
\path (A) edge node[below =0.15 cm] {} (C);
\path (B) edge node[below =0.15 cm] {} (C);
\path (C) edge node[below =0.15 cm] {} (E);
\path (D) edge node[below =0.15 cm] {} (G);
\path (F) edge node[below =0.15 cm] {} (G);
\path (G) edge node[below =0.15 cm] {} (B);
\path (J) edge node[below =0.15 cm] {} (H);
\path (J) edge node[below =0.15 cm] {} (I);
\path (H) edge node[below =0.15 cm] {} (A);
\path (H) edge node[below =0.15 cm] {} (F);
\path (I) edge node[below =0.15 cm] {} (D);
\end{tikzpicture}
\caption{The expression DAG $H^{(2)}$ obtained by transforming $H^{(1)}$
in figure \ref{pre_transform_dag}.}\label{post_transform_dag}
\end{figure}
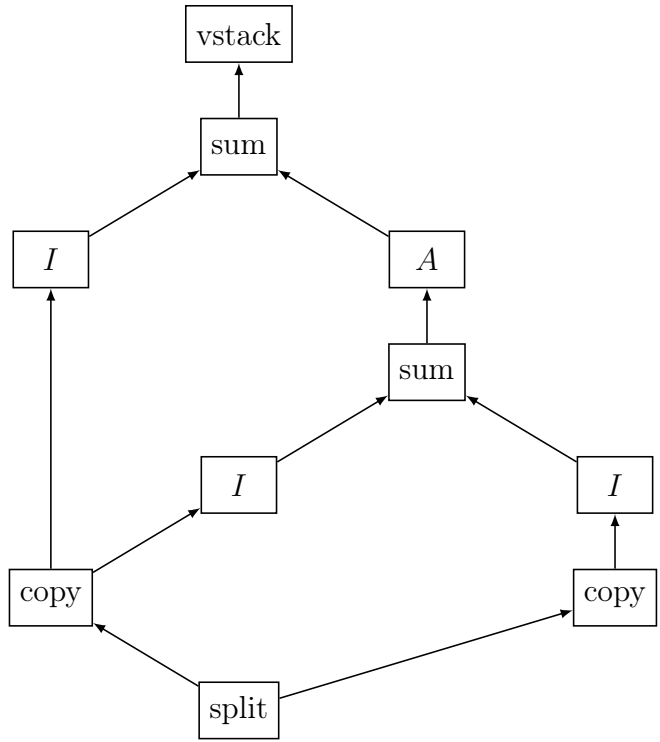

\paragraph{\texttt{Optimize-Graph}.}
The \texttt{Optimize-Graph} subroutine takes an FAO DAG $G$ as input
and outputs an equivalent FAO DAG $G^\mathrm{opt}$,
meaning that the output of algorithm
\ref{fao_dag_eval_alg} is the same for $G$ and $G^\mathrm{opt}$.
We choose $G^\mathrm{opt}$ by optimizing $G$ so that the runtime of
algorithm \ref{fao_dag_eval_alg} is as short as possible
(see \S\ref{opt_runtime}).
We also compress the FAO data and edge arrays to reduce the graph's
memory footprint (see \S\ref{opt_memory}).
We could optimize the graph for the adjoint, $G^*$, as well,
but asymptotically at least the flop count and memory footprint for
$G^*$ will be the same as for $G$,
meaning optimizing $G$ is the same as jointly optimizing $G$ and $G^*$.

\paragraph{\texttt{Matrix-Repr}.}
The \texttt{Matrix-Repr} subroutine takes a list of linear expression
DAGs, $(e_1,\ldots,e_\ell)$, and an ordering over the variables in the
expression DAGs, $<_{V}$, as input and outputs a sparse matrix.
Note that the input types are the same as in the
\texttt{Graph-Repr} subroutine.
In fact, for a given input the sparse matrix output by
\texttt{Matrix-Repr}
represents the same linear function as the FAO DAG output by
\texttt{Graph-Repr}.
The \texttt{Matrix-Repr} subroutine is used by the standard
canonicalization algorithm to produce a sparse matrix representation of
a cone program.
The implementation of \texttt{Matrix-Repr} is described in
\S\ref{sparse_mat_sec}.

\paragraph{Overall algorithm.}
With all the subroutines in place, the matrix-free canonicalization
algorithm is straightforward.
The implementation is given in algorithm \ref{mat_free_canon_alg}.

\begin{algorithm}
\caption{Matrix-free canonicalization.}\label{mat_free_canon_alg}
\begin{algorithmic}
\Require{$p$ is an OPR that satisfies the requirements of DCP.}
\Statex
\State $(s,o,C) \gets \texttt{Conic-Form}(p)$.
\State Choose any ordering $<_{V}$ on the variables in $(s,o,C)$.
\State Choose any ordering $<_{C}$ on the constraints in $C$.
\State $((e_1,\K_1),\ldots, (e_\ell,\K_\ell)) \gets$ the constraints in
$C$ ordered according to $<_{C}$.
\State $c_\mathrm{mat} \gets \texttt{Matrix-Repr}((\texttt{Linear}(o)), <_{V})$.
\State Convert $c_\mathrm{mat}$ from a 1-by-$n$ sparse matrix into a length $n$ array $c_\mathrm{arr}$.
\State $d_\mathrm{arr} \gets \texttt{Evaluate}(\texttt{Constant}(o))$.
\State $b_\mathrm{arr} \gets \vstack(
\texttt{Evaluate}(\texttt{Constant}(e_1)), \ldots,
\texttt{Evaluate}(\texttt{Constant}(e_\ell))
)$.
\State $G \gets \texttt{Graph-Repr}(
(\texttt{Linear}(e_1), \ldots, \texttt{Linear}(e_\ell)), <_{V})$.
\State $G^\mathrm{opt} \gets \texttt{Optimize-Graph}(G)$
\State $\K_\mathrm{list} \gets (\K_1, \ldots, \K_\ell)$.
\State \Return $(c_\mathrm{arr}, d_\mathrm{arr}, b_\mathrm{arr},
G^\mathrm{opt},\K_\mathrm{list})$.
\Statex
\end{algorithmic}
\end{algorithm}

\section{Numerical results}\label{numerical_sec}

\subsection{Implementation}

We have implemented the matrix-free canonicalization algorithm as an
extension of CVXPY \cite{cvxpy_paper}, available at
\begin{center}
    \url{https://github.com/SteveDiamond/cvxpy}.
\end{center}
To solve the resulting matrix-free cone programs,
we implemented modified versions of SCS \cite{SCSpaper} and
POGS \cite{fougner2015parameter}
that are truly matrix-free, available at
\begin{center}
    \url{https://github.com/SteveDiamond/scs}, \\
    \url{https://github.com/SteveDiamond/pogs}.
\end{center}
(The details of these modifications will be described in future work.)
Our implementations are still preliminary and can be improved in many ways.
We also emphasize that the canonicalization is independent of
the particular matrix-free cone solver used.

In this section we benchmark our implementation of matrix-free
canonicalization and of matrix-free SCS and POGS
on several convex optimization problems involving fast linear transforms.
We compare the performance of our matrix-free convex optimization
modeling system with that of the current CVXPY modeling system,
which represents the matrix $A$ in a cone program as a sparse matrix
and uses standard cone solvers.
The standard cone solvers and matrix-free SCS were run serially on a
single Intel Xeon processor,
while matrix-free POGS was run on a Titan X GPU.

\subsection{Nonnegative deconvolution}

We applied our matrix-free convex optimization modeling system to the
nonnegative deconvolution problem (\ref{nonneg_deconv_prob}).
The Python code below constructs and solves problem
(\ref{nonneg_deconv_prob}).
The constants $c$ and $b$ and problem size $n$ are defined elsewhere.
The code is only a few lines,
and it could be easily modified to add regularization on $x$ or apply a
different cost function to $c*x - b$.
The modeling system would automatically adapt to solve the modified
problem.

\begin{verbatim}
# Construct the optimization problem.
x = Variable(n)
cost = sum_squares(conv(c, x) - b)
prob = Problem(Minimize(cost),
               [x >= 0])
# Solve using matrix-free SCS.
prob.solve(solver=MAT_FREE_SCS)
\end{verbatim}

\paragraph{Problem instances.}
We used the following procedure to generate interesting (nontrivial)
instances of problem (\ref{nonneg_deconv_prob}).
For all instances the vector $c \in \reals^n$ was a Gaussian kernel
with standard deviation $n/10$.
All entries of $c$ less than $10^{-6}$ were set to $10^{-6}$,
so that no entries were too close to zero.
The vector $b \in \reals^{2n-1}$ was generated by picking a solution
$\tilde x$ with 5 entries randomly chosen to be nonzero.
The values of the nonzero entries were chosen uniformly at random from
the interval $[0,n/10]$.
We set $b = c*\tilde{x} + v$,
where the entries of the noise vector $v \in \reals^{2n-1}$ were
drawn from a normal distribution with mean zero and variance
$\|c*\tilde{x}\|^2/(400(2n-1))$.
Our choice of $v$ yielded a signal-to-noise ratio near 20.

\begin{figure}[t]
\begin{center}
\includegraphics[width=0.7\textwidth]{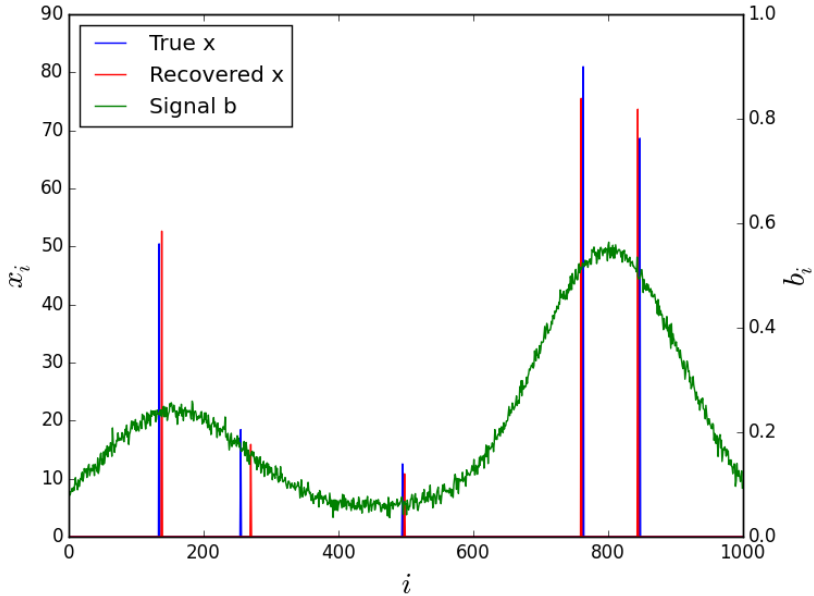}
\end{center}
\caption{Results for a problem instance with $n=1000$.
}\label{nonneg_deconv_result}
\end{figure}

While not relevant to solving the optimization problem, the solution
of the nonnegative deconvolution problem often, but not always,
(approximately) recovers the original vector $\tilde x$.
Figure \ref{nonneg_deconv_result} shows the solution recovered
by ECOS \cite{bib:Domahidi2013ecos} for a problem instance with
$n=1000$.
The ECOS solution $x^\star$ had a cluster of 3-5 adjacent nonzero
entries around each spike in $\tilde x$.
The sum of the entries was close to the value of the spike.
The recovered $x$ in figure \ref{nonneg_deconv_result} shows only the
largest entry in each cluster, with value set to the sum of the
cluster's entries.

\begin{figure}[t]
\begin{center}
\includegraphics[width=0.7\textwidth
]{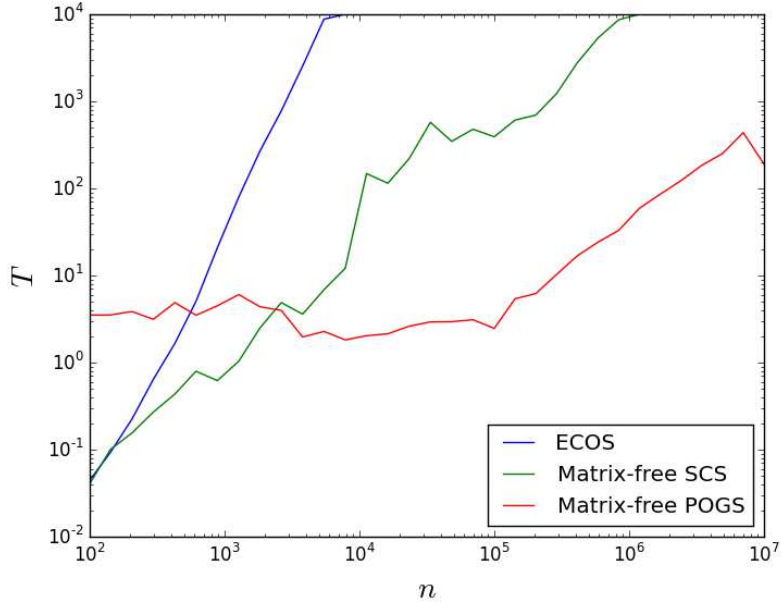}
\end{center}
\caption{Solve time in seconds $T$ versus variable size $n$.
}\label{nonneg_deconv_solve_times}
\end{figure}

\paragraph{Results.}
Figure \ref{nonneg_deconv_solve_times} compares the performance on
problem (\ref{nonneg_deconv_prob}) of the interior-point solver
ECOS \cite{bib:Domahidi2013ecos}
and matrix-free versions of SCS and POGS
as the size $n$ of the optimization variable increases.
We limited the solvers to $10^4$ seconds.

For each variable size $n$ we generated ten different problem instances
and recorded the average solve time for each solver.
ECOS and matrix-free SCS were run with an absolute and relative
tolerance of $10^{-3}$ for the duality gap,
$\ell_2$ norm of the primal residual, and $\ell_2$ norm of the dual
residual.
Matrix-free POGS was run with an absolute tolerance of $10^{-4}$ and a
relative tolerance of $10^{-3}$.

The slopes of the lines show how the solvers scale.
The least-squares linear fit for the ECOS solve times has slope
$3.1$, which indicates that the solve time scales like
$n^3$, as expected.
The least-squares linear fit for the matrix-free SCS solve times has
slope $1.3$,
which indicates that the solve time scales like the expected $n\log n$.
The least-squares linear fit for the matrix-free POGS solve times in
the range $n \in [10^5,10^7]$ has slope $1.1$,
which indicates that the solve time scales like the expected $n\log n$.
For $n < 10^5$, the GPU overhead
(launching kernels, synchronization, \etc)
dominates, and the solve time is nearly constant.




\subsection{Sylvester LP}

We applied our matrix-free convex optimization modeling system to
Sylvester LPs, or convex optimization problems of the form
\begin{equation}\label{sylvester_lp}
\begin{array}{ll}
\mbox{minimize} & \Tr(D^TX) \\
\mbox{subject to} & AXB \leq C \\
& X \geq 0,
\end{array}
\end{equation}
where $X \in \reals^{p \times q}$ is the optimization variable,
and $A \in \reals^{p \times p}$, $B \in \reals^{q \times q}$,
$C \in \reals^{p \times q}$, and $D \in \reals^{p \times q}$
are problem data.
The inequality $AXB \leq C$ is a variant of the
Sylvester equation $AXB = C$ \cite{Gardiner:1992:SSM:146847.146929}.

Existing convex optimization modeling systems will convert
problem (\ref{sylvester_lp}) into the vectorized format
\begin{equation}\label{vect_sylvester_lp}
\begin{array}{ll}
\mbox{minimize} & \vect(D)^T\vect(X) \\
\mbox{subject to} & (B^T \kronecker A)\vect(X) \leq \vect(C) \\
& \vect(X) \geq 0,
\end{array}
\end{equation}
where $B^T \kronecker A \in \reals^{pq \times pq}$ is the Kronecker
product of $B^T$ and $A$.
Let $p = kq$ for some fixed $k$, and let $n=kq^2$ denote the
size of the optimization variable.
A standard interior-point solver will take $O(n^3)$ flops and
$O(n^2)$ bytes of memory to solve problem (\ref{vect_sylvester_lp}).
A specialized matrix-free solver that exploits the matrix product $AXB$,
by contrast, can solve problem (\ref{sylvester_lp}) in $O(n^{1.5})$
flops using $O(n)$ bytes of memory \cite{VaB:95}.

\paragraph{Problem instances.}
We used the following procedure to generate interesting (nontrivial)
instances of problem (\ref{sylvester_lp}).
We fixed $p = 5q$ and generated $A$ and $B$ by first drawing entries
i.i.d.\ from the folded standard normal distribution
(\ie, the absolute value of the standard normal distribution).
We then added $10^{-6}$ to all entries of $A$ and $B$ so they were
guaranteed to be positive.
We generated $D$ by drawing entries i.i.d.\ from a standard normal
distribution.
We fixed $C = 11^T$.
Our method of generating the problem data ensured the problem was
feasible and bounded.

\begin{figure}
\begin{center}
\includegraphics[width=0.7\textwidth]{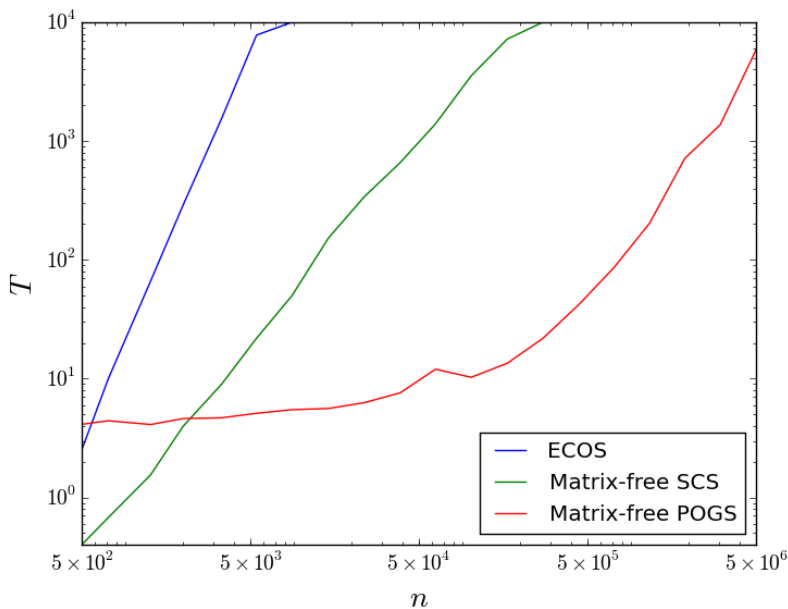}
\end{center}
\caption{Solve time in seconds $T$ versus variable size $n$.
}\label{sylvester_solve_times}
\end{figure}

\paragraph{Results.}
Figure \ref{sylvester_solve_times} compares the performance on
problem (\ref{sylvester_lp}) of the interior-point solver
ECOS \cite{bib:Domahidi2013ecos}
and matrix-free versions of SCS and POGS as the size
$n = 5q^2$ of the optimization variable increases.
We limited the solvers to $10^4$ seconds.
For each variable size $n$ we generated ten different problem instances
and recorded the average solve time for each solver.
ECOS and matrix-free SCS were run with an absolute and relative
tolerance of $10^{-3}$ for the duality gap,
$\ell_2$ norm of the primal residual, and $\ell_2$ norm of the dual
residual.
Matrix-free POGS was run with an absolute tolerance of
$10^{-4}$ and a relative tolerance of $10^{-3}$.
The slope of the lines show how the solvers scale.
The least-squares linear fit for the ECOS solve times has slope $3.3$,
which indicates that the solve time scales like $n^3$, as expected.
The least-squares linear fit for the matrix-free SCS solve times has
slope $1.7$,
which indicates that the solve time scales like the expected $n^{1.5}$.
The least-squares linear fit for the matrix-free POGS solve times in
the range $n \in [10^5,5 \times 10^6]$ has slope $1.6$,
which indicates that the solve time scales like the expected $n^{1.5}$.
For $n < 10^5$, the GPU overhead (launching kernels, synchronization,
\etc) dominates,
and the solve time is nearly constant.

\section*{Acknowledgements}
We would like to thank Eric Chu, Michal Kocvara, and Alex Aiken
for helpful comments on earlier versions of this work,
and to thank Chris Fougner, John Miller, Jack Zhu, and Paul Quigley
for their work on the POGS cone solver and CVXcanon \cite{CVXcanon},
which both contributed to the implementation of matrix-free CVXPY.
This material is based upon work supported by the
National Science Foundation Graduate
Research Fellowship under Grant No. DGE-114747
and by the DARPA X-DATA program.

\newpage
\appendix

\section{Equivalence of the cone program}\label{equivalence_sec}

In this section we explain the precise sense in which the cone program
output by the matrix-free canonicalization algorithm is equivalent to
the original convex optimization problem.

\begin{theorem}\label{equiv_theorem}
Let $p$ be a convex optimization problem whose OPR is a valid input to
the matrix-free canonicalization algorithm.
Let $\Phi(p)$ be the cone program represented by the output of the
algorithm given $p$'s OPR as input.
All the variables in $p$ are present in $\Phi(p)$,
along with new variables introduced during the canonicalization process
\cite{Grant2004,cvxjl}.
Let $x \in \reals^n$ represent the variables in $p$ stacked into a vector and
$t \in \reals^m$ represent the new variables in $\Phi(p)$ stacked into a vector.

The problems $p$ and $\Phi(p)$ are equivalent in the following sense:
\begin{enumerate}
  \item\label{primal_equal} For all $x$ feasible in $p$, there exists $t^\star$ such that $(x, t^\star)$ is feasible in $\Phi(p)$ and $p(x) = \Phi(p)(x, t^\star)$.
  \item\label{primal_bounded} For all $(x,t)$ feasible in $\Phi(p)$, $x$ is feasible in $p$ and $p(x) \leq \Phi(p)(x, t)$.
\end{enumerate}
For a point $x$ feasible in $p$, by $p(x)$ we mean the value of $p$'s
objective evaluated at $x$.
The notation $\Phi(p)(x,t)$ is similarly defined.
\end{theorem}

\begin{proof} See \cite{Grant2004}.
\end{proof}

Theorem \ref{equiv_theorem} implies that $p$ and $\Phi(p)$ have the
same optimal value.
Moreover, $p$ is infeasible if and only if $\Phi(p)$ is infeasible,
and $p$ is unbounded if and only if $\Phi(p)$ is unbounded.
The theorem also implies that any solution $x^\star$ to $p$ is part of
a solution $(x^\star,t^\star)$ to $\Phi(p)$ and vice versa.

A similar equivalence holds between the Lagrange duals of $p$ and $\Phi(p)$,
but the details are beyond the scope of this paper.
See \cite{Grant2004} for a discussion of the dual of the cone program
output by the canonicalization algorithm.

\section{Sparse matrix representation}\label{sparse_mat_sec}

In this section we explain the \texttt{Matrix-Repr} subroutine used in
the standard canonicalization algorithm to obtain a sparse matrix
representation of a cone program.
Recall that the subroutine takes a list of linear expression DAGs,
$(e_1,\ldots,e_\ell)$,
and an ordering over the variables in the expression DAGs, $<_{V}$,
as input and outputs a sparse matrix $A$.

The algorithm to carry out the subroutine is not discussed anywhere in the literature,
so we present here the version used by CVXPY \cite{cvxpy_paper}.
The algorithm first converts each expression DAG into a map from
variables to sparse matrices,
representing a sum of terms.
For example, if the map $\phi$ maps the variable $x \in \reals^n$ to the sparse matrix coefficient $B \in \reals^{m \times n}$
and the variable $y \in \reals^n$ to the sparse matrix coefficient $C \in \reals^{m \times n}$,
then $\phi$ represents the sum $Bx + Cy$.

The conversion from expression DAG to map of variables to sparse
matrices is done using algorithm \ref{expr_dag_to_mat_alg}.
The algorithm uses the subroutine \texttt{Matrix-Coeff},
which takes a node representing a linear function $f$ and indices $i$
and $j$ as inputs and outputs a sparse matrix $D$.
Let $\tilde f$ be a function defined on the range of $f$'s
$i$th input such that $\tilde f(x)$ is equal to $f$'s
$j$th output when $f$ is evaluated on $i$th input $x$ and
zero-valued matrices
(of the appropriate dimensions) for all other inputs.
The output of \texttt{Matrix-Coeff} is the sparse matrix $D$ such
that for any value $x$ in the domain of $\tilde f$,
\[
D\vect(x) = \vect(\tilde f(x)).
\]

The sparse matrix coefficients in the maps of variables to sparse matrices
are assembled into a single sparse matrix $A$, as follows:
Let $x_1,\ldots,x_k$ be the variables in the expression DAGs,
ordered according to $<_V$.
Let $n_i$ be the length of $x_i$ if the variable is a vector and of
$\vect(x_i)$ if the variable is a matrix, for $i=1,\ldots,k$.
Let $m_j$ be the length of expression DAG $e_j$'s output,
for $j=1,\ldots,\ell$.
The coefficients for $x_1$ are placed in the first $n_1$ columns in $A$,
the coefficients for $x_2$ in the next $n_2$ columns, \etc
~Similarly, the coefficients from $e_1$ are placed in the first $m_1$ rows of $A$,
the coefficients from $e_2$ in the next $m_2$ rows, \etc

\begin{algorithm}
\caption{Convert an expression DAG into a map from variables to sparse matrices.}\label{expr_dag_to_mat_alg}
\begin{algorithmic}
\Require{$e$ is a linear expression DAG that outputs a single vector.}
\Statex
\State Create an empty queue $Q$ for nodes that are ready to evaluate.
\State Create an empty set $S$ for nodes that have been evaluated.\State Create a map $M$ from (node, output index) tuples to maps of variables to sparse matrices.
\For {every start node $u$ in $e$}
    \State $x \gets$ the variable represented by node $u$.
    \State $n \gets$ the length of $x$ if the variable is a vector and of $\vect(x)$ if the variable is a matrix.
    \State $M[(u,1)] \gets$ a map with key $x$ and value the $n$-by-$n$ identity matrix.
    \State Add $u$ to $S$.
\EndFor
\State Add all nodes in $e$ to $Q$ whose only incoming edges are from start nodes.
\While {$Q$ is not empty}
    \State $u \gets \text{pop the front node of } Q$.
    \State Add $u$ to $S$.
    \For {edge $(u,p)$ in $u$'s $E_\mathrm{out}$, with index $j$}
        \State Create an empty map $M_j$ from variables to sparse matrices.
        \For {edge $(v,u)$ in $u$'s $E_\mathrm{in}$, with index $i$}
            \State $A^{(ij)} \gets \texttt{Matrix-Coeff}(u, i, j)$.
            \State $k \gets$ the index of $(v,u)$ in $v$'s $E_\mathrm{out}$.
            \For {key $x$ and value $C$ in $M[(v,k)]$}
                \If {$M_j$ has an entry for $x$}
                    \State $M_j[x] \gets M_j[x] + A^{(ij)}C$.
                \Else
                    \State $M_j[x] \gets A^{(ij)}C$.
                \EndIf
            \EndFor
        \EndFor
        \State $M[(u,j)] \gets M_j$.
        \If {for all edges $(q,p)$ in $p$'s $E_\mathrm{in}$,
             $q$ is in $S$}
            \State Add $p$ to the end of $Q$.
        \EndIf
    \EndFor
\EndWhile
\State $u_\mathrm{end} \gets$ the end node of $e$.
\State \Return $M[(u_\mathrm{end},1)]$.
\Statex
\end{algorithmic}
\end{algorithm}

\newpage
\bibliography{abs_ops}

\end{document}